\documentclass[12pt]{amsart}

\usepackage[utf8]{inputenc}
\usepackage{amsmath,amsthm,amssymb}
\usepackage{indentfirst}
\usepackage{url}
\usepackage{tikz-cd}
\usepackage[colorlinks=true]{hyperref}
\usepackage{colonequals}
\usepackage{graphicx}
\usepackage{enumitem}
\usepackage{soul}
\usepackage{stmaryrd}
\usepackage{mathrsfs,mathtools}
\usepackage{tgpagella}
\usepackage{bbm}
\usepackage{subcaption}
\usepackage[alphabetic]{amsrefs}
\usepackage{xcolor}
\usepackage[mathscr]{eucal}

\usepackage[margin=1in]{geometry}

\newtheorem{thm}{Theorem}[section]
\newtheorem{lem}[thm]{Lemma}
\newtheorem{conj}[thm]{Conjecture}
\newtheorem{prop}[thm]{Proposition}

\newtheorem{cor}[thm]{Corollary}
\newtheorem{ques}[thm]{Question}
{
\theoremstyle{definition}
\newtheorem{remx}[thm]{Remark}

\newtheorem{defnx}[thm]{Definition}
\newtheorem{examplex}[thm]{Example}
}
\definecolor{gold}{RGB}{255,215,0}

\newenvironment{defn}
{\pushQED{\qed}\defnx}
{\popQED\enddefnx}
\newenvironment{example}
{\pushQED{\qed}\examplex}
{\popQED\endexamplex}

\newcommand{\ZZ}{\mathbb Z}

\newcommand{\CC}{\mathbb C}

\newcommand{\supp}{\mathrm{supp}}
\newcommand{\Newton}{\mathrm{Newton}}
\newcommand{\wt}{\mathrm{wt}}

\newcommand{\SBD}{\mathcal{SBD}}

\newcommand{\MBPD}{\mathrm{MBPD}}

\newcommand{\GL}{\mathrm{GL}}

\title{Supports of Castelnuovo-Mumford polynomials}\date{}

\author{Elena S.~Hafner}

\address{Elena S.~Hafner, Department of Mathematics, University of Washington, Seattle, WA 98195-4350. \newline\textup{eshafner@uw.edu}
}

\thanks{Elena S.~Hafner received support from NSF Grant DMS-2402150. }

\begin{document}

\begin{abstract}
    The Castelnuovo-Mumford polynomials are the maximal degree components of Grothendieck polynomials.  The support of each Castelnuovo-Mumford polynomial is conjectured to be M-convex, i.e. the set of integer points of a generalized permutahedron (M{\'e}sz{\'a}ros and St.~Dizier, 2020).  This conjecture is known to hold in certain special cases but remains open in general. We define new families of permutations whose Castelnuovo-Mumford polynomials we show to have M-convex support.  Specifically, we investigate which permutations have Castelnuovo-Mumford polynomials whose supports are the set of integer points in a schubitope.
\end{abstract}
\maketitle

\section{Introduction}
Grothendieck polynomials are a family of polynomials, indexed by permutations, which serve as representatives of the K-theoretic classes of the complete flag variety \cite{lascoux1982structure}.  They can be thought of as non-homogeneous analogues of Schubert polynomials.  For a given permutation $w$, the Schubert polynomial $\mathfrak{S}_w$ is the lowest degree homogeneous component of the Grothendieck polynomial $\mathfrak{G}_w$.  Though the combinatorics of Schubert polynomials has been well studied over the years, there has more recently been a surge of study related to the combinatorics of Grothendieck polynomials and especially, their maximal degree homogeneous components \cite{chou2024constructing,chou2025grothendieck,dreyer2024degree,Hafner22BPDs,Bubbling,pan2024top,psw2021,rajchgot2021degrees,rajchgot2023castelnuovo}.  Pechenik, Speyer, and Weigandt gave these maximal degree components the name \textbf{Castelnuovo-Mumford polynomials} in \cite{psw2021} because the Castelnuovo-Mumford regularity of a matrix Schubert variety can be computed as the difference between the highest and lowest degree homogeneous components of the Grothendieck polynomial.

Huh, Matherne, M{\'e}sz{\'a}ros, and St.~Dizier \cite{hmms2022} conjectured that the supports of homogenized Grothendieck polynomials are M-convex, i.e. the support of a homogenized Grothendieck polynomial is precisely the set of integer points of some generalized permutahedron.  In particular, this broadens an earlier conjecture that every homogeneous component of $\mathfrak{G}_w$ has M-convex support \cite{ms2020}.  It is currently known that homogenized Grothendieck polynomials have M-convex supports for certain families of permutations including: vexillary permutations \cite{Bubbling}, fireworks permutations \cite{chou2025newton}, and permutations whose Schubert polynomials have all nonzero coefficients equal to $1$ \cite{ccmm2022}.  This conjecture was also proven for certain subsets of vexillary permutations in \cite{ey2017} and \cite{ms2020}.  In \cite{Bubbling}, Hafner, M{\'e}sz{\'a}ros, Setiabrata, and St.~Dizier proved M-convexity for vexillary Grothendieck polynomials by constructing a set of diagrams, known as streamlined bubbling diagrams, whose weights correspond to the supports of $\mathfrak{G}_w$.  As a consequence, they showed that for any vexillary permutation, the support of the Castelnuovo-Mumford polynomial is the set of integer points of a particular type of generalized permutahedron: a schubitope.  They observe, however, that the convex hull of the support is \textit{not} always a schubitope for Castelnuovo-Mumford polynomials corresponding to non-vexillary permutations.  Inspired by this observation, we investigate the following questions.
\begin{ques}
    \label{q:schubitope}
    For which permutations is the support of the Castelnuovo-Mumford polynomial the set of integer points of a schubitope?
\end{ques}
\begin{ques}
    \label{q:bubbling}
    For which permutations can we construct a set of streamlined bubbling diagrams whose weights precisely compute the support of the Grothendieck polynomial? 
\end{ques}

We prove a simplified version of the strealined bubbling diagram model for vexillary Grothendieck polynomials.  Leveraging this simplified model, we prove the following result which was conjectured by Hafner, M{\'e}sz{\'a}ros, Setiabrata, and St.~Dizier.  Let $\hat{\mathfrak{G}}_w$ denote the Castelnuovo-Mumford polynomial for a permutation $w$, and let $\chi_{D}$ denote the dual character of the flagged Weyl module associated to a diagram $D$.
\begin{thm} \label{thm:TopVexChar}
    For any vexillary permutation $w$, $\hat{\mathfrak{G}}_w$ equals an integer multiple of $\chi_{D_{\text{top}}(w)}$.
\end{thm}
\noindent Here $D_{\text{top}}(w)$ is a particular diagram which we will define in Section~\ref{sec:BubblingDiagrams}. 

Building on the results for vexillary permutations, we define new families of permutations which we call almost vexillary permutations and dominant fireworks-vexillary chains.  We show that for any almost vexillary permutation or dominant fireworks-vexillary chain, the support of the Castelnuovo-Mumford polynomial is indeed the set of integer points of a schubitope.  Specifically, we prove the following.
\begin{thm}
\label{thm:topdeg-almostvex}
    Let $w$ be any almost vexillary permutation.  Then $\hat{\mathfrak{G}}_w$ is a scalar multiple of $\chi_{{D}}$ for some diagram ${D}$.  In particular, $\text{Supp}(\hat{\mathfrak{G}}_w)$ is M-convex, and $\text{Newton}(\hat{\mathfrak{G}}_w)$ is a schubitope.  
\end{thm}
\begin{thm}
    \label{thm:chains}
    If $w\in S_n$ is a dominant fireworks-vexillary chain, then $\supp(\mathfrak G_w)$ is precisely the set $\{\wt(\mathcal D)\colon \mathcal D \in \SBD(D(w), \emptyset,A)\}$ of weights of streamlined bubbling diagrams for some subset $A\subset D(w)$.  In particular, the homogenized Grothendieck polynomial $\widetilde{\mathfrak{G}}_w$ has M-convex support, and $\text{Newton}(\hat{\mathfrak{G}}_w)$ is a schubitope.
\end{thm}

\textbf{Outline of the paper.} In Section~\ref{sec:BG}, we review the relevant definitions and background results.  In Section~\ref{sec:BubblingDiagrams}, we provide an overview of the bubbling diagram model for supports of vexillary Grothendieck polynomials.  We prove a simplified version of these bubbling diagrams in Section~\ref{sec:NewvexBubb} and use it to prove Theorem~\ref{thm:TopVexChar}.  In Section~\ref{sec:extendedBubb}, we define almost vexillary and dominant fireworks-vexillary chain permutations and prove Theorems~\ref{thm:topdeg-almostvex} and~\ref{thm:chains}.  More generally, we provide methods by which permutations whose Castelnuovo-Mumford polynomials have schubitope support can be combined into additional permutations with this property.

\section{Background}
\label{sec:BG}
\subsection{Permutations}
We will be concerned with Grothendieck polynomials corresponding to particular families of permutations.  A \textbf{dominant permutation} is one that is $132$-avoiding, and a \textbf{vexillary permutation} is one that is $2143$-avoiding.  A permutation is called \textbf{fireworks} if the initial terms of the decreasing runs in its one-line notation form an increasing sequence.

\subsection{Grothendieck and Lascoux Polynomials}
For any $i$, the \textbf{divided difference operator} $\partial_i$ is defined by $\partial_i(f)=\frac{f-s_if}{x_i-x_{i+1}}$ where $f$ is a polynomial in $x_1,\ldots,x_{n+1}$ and $s_i$ acts on $f$ by swapping $x_i$ and $x_{i+1}$.  Then for any permutation $w \in S_n$, the \textbf{Grothendieck polynomial} $\mathfrak{G}_w$ is defined by
\[\mathfrak{G}_w=\begin{cases}
    x_1^{n-1} x_2^{n-2} \cdots x_{n-1}, & \text{if } w=n \text{  } n-1 \text{  } n-2 \text{  } \ldots \text{  } 1 \\
    \partial_i((1- x_{i+1})\mathfrak{G}_{w s_i}), & \text{if } w(i)<w({i+1})
\end{cases}.\]
Let $\mathfrak{G}_w^{(k)}$ denote the degree $k$ homogeneous component of $\mathfrak{G}_w$, and let $\hat{\mathfrak{G}}_w$ denote the highest degree component of $\mathfrak{G}_w$.  Pechenik, Speyer, and Weigandt \cite{psw2021} refer to $\hat{\mathfrak{G}}_w$ as the \textbf{Castelnuovo-Mumford polynomial}.  The lowest degree component of $\mathfrak{G}_w$ is the Schubert polynomial $\mathfrak{S}_w$.  We denote by $\widetilde{\mathfrak G}_w$ the \textbf{homogenized Grothendieck polynomial}
\[
\widetilde{\mathfrak G}_w(x_1, \dots, x_n, z) := \sum_{k=\deg(\mathfrak{S}_w)}^{\deg(\mathfrak G_w)}(-1)^{k-\deg(\mathfrak{S}_w)}\mathfrak G_w^{(k)}(x_1, \dots, x_n)z^{\deg(\mathfrak G_w) - k}.
\]

Similarly, for any weak composition $\alpha$, the \textbf{Lascoux polynomial} $\mathfrak{L}_{\alpha}$ is given by
\[\mathfrak{L}_{\alpha}=\begin{cases}
    x^{\alpha}, & \text{if } \alpha \text{ is weakly decreasing}\\
    \partial_i(x_i((1- x_{i+1})\mathfrak{L}_{s_i \alpha})), & \text{if } \alpha_i < \alpha_{i+1}
\end{cases}.\]
Let $\hat{\mathfrak{L}}_{\alpha}$ denote the highest degree component of $\mathfrak{L}_{\alpha}$.  We will refer to $\hat{\mathfrak{L}}_{\alpha}$ as the \textbf{top-Lascoux} polynomial.

\subsection{Supports of Polynomials and M-Convexity}
\label{sec:MConvDefs}
Given a polynomial $f$ in $x_1,\ldots,x_n$, the \textbf{support} is the set $\supp (f)$ of all $a=(a_1,\ldots,a_n)$ such that $x^a$ has nonzero coefficient in $f$.
\begin{defnx}
\label{defn:NewtonPolytope}
The \textbf{Newton Polytope} of a polynomial $f$ in variables $x_1,\ldots,x_n$ is the convex hull of the support of $f$, denoted $\Newton (f)$.  A polynomial $f$ is said to have \textbf{saturated Newton polytope (SNP)} if $\supp (f)$ is precisely the set of integer points of $\Newton (f)$.
\end{defnx}

A set $S\subset \mathbb{Z}^n$ is defined to be \textbf{M-convex} if for any pair of elements $x,y \in S$ and any $i\in[n]$ for which $x_i > y_i$, there exists an index $j\in[n]$ such that $x_j < y_j$ and $x - e_i + e_j, y - e_j + e_i \in S$.  The convex hull of any M-convex set is a generalized permutahedron, and the set of integer points in any generalized permutahedron is M-convex \cite{Murota}.  In particular, if a polynomial has M-convex support, then it has SNP.

\subsection{Diagrams}
A \textbf{diagram} is a subset $D$ of the $n\times m$ grid.  Given a diagram $D$, let $D_i$ denote the $i^{th}$ column of $D$.  The \textbf{weight} of $D$ is the vector $\wt({D})=(\wt({D})_1, \ldots, \wt({D})_n)$ where $\wt(D)_i$ is the number of squares in row $i$ of $D$.

We say that two sets $R,S \subset [n]$ satisfy $R \leq S$ if $|R| = |S|$ and for every $k$, the $k$-th smallest element of $R$ is less than or equal to the $k$-th smallest element of $S$.  For two diagrams $C,D \subset [n] \times [m]$, we say $C\leq D$ if $C_j \leq D_j$ for every $j \in [m]$.

For any $w \in S_n$, the \textbf{Rothe diagram} $D(w)$ is the subset
\[D(w):= \{(i,j)\in[n]\times[n]\colon i < w^{-1}(j) \textup{ and } j < w(i)\}\]
of the $n \times n$ grid.  Pictorially, the Rothe diagram can be thought of in the following way:
\begin{enumerate}
    \item Beginning with the $n \times n$ grid, place a dot in position $(i,w(i))$ for every $i \in [n]$.
    \item Starting from each of these dots, draw a horizontal line to the right edge of the grid and a vertical line to the bottom edge of the grid.  The squares which remain empty after this process are the Rothe diagram.
\end{enumerate}

\begin{figure}
    \centering
    \begin{tikzpicture}
    \filldraw[draw=darkgray, color=green, opacity=0.5] (.5,3) rectangle (3.5,3.5);
    \filldraw[draw=darkgray, color=green, opacity=0.5] (1,2) rectangle (3,2.5);
    \filldraw[draw=darkgray, color=green, opacity=0.5] (1.5,.5) rectangle (2,1.5);
    \draw[step=.5cm,gray,very thin] (0,0) grid (4,4);
    \node at (.75,3.25) {1};
    \node at (1.25,3.25) {1};
    \node at (1.75,3.25) {1};
    \node at (2.25,3.25) {1};
    \node at (2.75,3.25) {1};
    \node at (3.25,3.25) {1};
    \node at (1.25,2.25) {2};
    \node at (1.75,2.25) {2};
    \node at (2.25,2.25) {2};
    \node at (2.75,2.25) {2};
    \node at (1.75,1.25) {3};
    \node at (1.75,.75) {3};
    \draw[black,thick] (.25,0) -- (.25,3.5) .. controls (.25,3.75) .. (.5,3.75) -- (4,3.75);
    \draw[black,thick] (.75,0) -- (.75, 2.5) .. controls (.75,2.75) .. (1,2.75) -- (4,2.75);
    \draw[black,thick] (1.25,0) -- (1.25,1.5) .. controls (1.25,1.75) .. (1.5,1.75) -- (4,1.75);
    \draw[black,thick] (1.75,0) .. controls (1.75,.25) .. (2,.25) -- (4,.25);
    \draw[black,thick] (2.25,0) -- (2.25,1) .. controls (2.25,1.25) .. (2.5,1.25) -- (4,1.25);
    \draw[black,thick] (2.75,0) -- (2.75,.5) .. controls (2.75,.75) .. (3,.75) -- (4,.75);
    \draw[black,thick] (3.25,0) -- (3.25,2) .. controls (3.25,2.25) .. (3.5,2.25) -- (4,2.25);
    \draw[black,thick] (3.75,0) -- (3.75,3) .. controls (3.75,3.25) .. (4,3.25);
    \end{tikzpicture}
    \caption{Rothe diagram for $w= 18273564$.  Numbers in each cell indicate the value of the rank function $r_{D(w)}(i,j)$.}
    \label{fig:Rothe-example}
\end{figure}
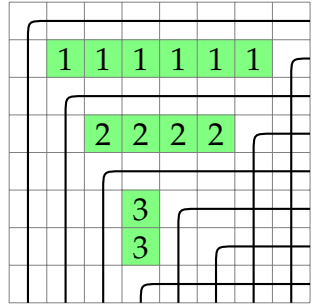

See Figure \ref{fig:Rothe-example} for an example.  Rothe diagrams are equipped with a rank function $r_{D(w)}\colon D(w)\to\ZZ_{\geq 0}$ defined by
\[r_{D(w)}(i,j) \colonequals |\{(k,w(k))\colon k < i \textup{ and } w(k) < j\}|.\]
One can think of this rank function as counting the number of hooks which pass northwest of the square $(i,j)$ in the pictorial construction above.

For any weak composition $\alpha=(\alpha_1,\ldots ,\alpha_n)$, there is an associated \textbf{skyline diagram} $D_{Sky}(\alpha)$ defined by 
\[D_{Sky}(\alpha):=\{(i,j)\colon 1\leq j\leq\alpha_i\}.\]
In other words, $D_{Sky}(\alpha)$ is the left-justified diagram containing $\alpha_i$ squares in each row $i$.  See Figure~\ref{fig:skyline} for an example.

\begin{figure}
    \centering
    \begin{tikzpicture}
    \filldraw[draw=darkgray, color=green, opacity=0.5] (0,3) rectangle (3,3.5);
    \filldraw[draw=darkgray, color=green, opacity=0.5] (0,2) rectangle (2,2.5);
    \filldraw[draw=darkgray, color=green, opacity=0.5] (0,.5) rectangle (.5,1.5);
    \draw[step=.5cm,gray,very thin] (0,0) grid (4,4);
    \end{tikzpicture}
    \caption{Skyline diagram for $\alpha=(0,6,0,4,0,1,1,0)$}
    \label{fig:skyline}
\end{figure}
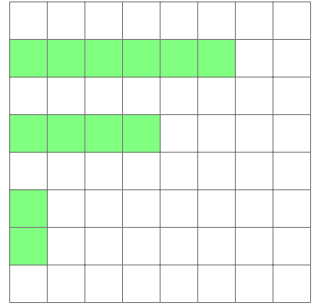

A diagram is said to be \textbf{\%-avoiding} if for any $i<i',j'<j$ such that $(i,j),(i',j')\in D$, either $(i,j')\in D$ or $(i',j) \in D$.  Note that Rothe diagrams and skyline diagrams are always \%-avoiding.





\subsection{Flagged Weyl modules, Dual Characters, and Schubitopes}
\label{sec:DualChar}
For a matrix $M \in M_n(\CC)$ and $R,S\subseteq[n]$, let $M_R^S$ denote the submatrix of $M$ obtained by restricting to rows $S$ and columns $R$.  Let $\mathfrak b$ denote the set of upper triangular matrices in $M_n(\CC)$, and let $B$ denote the set of invertible upper triangular matrices in $\GL_n(\CC)$.  

Let $Y$ be an upper triangular matrix with indeterminates $y_{ij}$ in the positions $i \leq j$ and zeroes elsewhere.  Let $\CC[\mathfrak b]$ denote the polynomial ring in the variables $\{y_{ij}\colon i \leq j\}$. The group $B$ acts on $\CC[\mathfrak b]$ on the right via $f(Y)\cdot b\colonequals f(b^{-1}Y)$.

\begin{defnx}
    Fix a diagram $D\subseteq[n]\times[m]$.  The \textbf{flagged Weyl module} of $D$ is 
\[
\mathcal M_D :=\mathrm{Span}_\CC\left\{\prod_{j=1}^m\det\left(Y_{D_j}^{C_j}\right)\colon C\leq D\right\}.
\]
\end{defnx}

The \textbf{character} of any $B$-module $N$ is given by 
\[\mathrm{char}_N(\mathrm{diag}(x_1, \dots, x_n)) = \mathrm{tr}(\mathrm{diag}(x_1,\dots, x_n)\colon N \to N)\]
where $\mathrm{diag}(x_1, \dots, x_n)$ is the diagonal matrix with entries $x_1, \dots, x_n$, viewed as a linear map from $N$ to $N$ via the $B$-action.  The \textbf{dual character} of $N$ is 
\[\mathrm{char}_N^*(\mathrm{diag}(x_1, \dots, x_n)) = \mathrm{tr}(\mathrm{diag}(x_1^{-1},\dots, x_n^{-1})\colon N \to N).\]

We will denote the dual character of $\mathcal M_D$ by $\chi_D\colonequals\mathrm{char}_{\mathcal M_D}^*$.

\begin{prop}[{\cite{fink2018schubert}*{Theorem 7}}]
\label{prop:support-of-chi}
The dual character $\chi_D$ of $\mathcal M_D$ is a polynomial in $\ZZ[x_1, \dots, x_n]$ whose support is $\{\wt(C)\colon C\leq D\}$.
\end{prop}



Originally defined by Monical, Tokcan, and Yong \cite{mty2019}, schubitopes are a family of generalized permutahedra associated to diagrams.  In this article, we will use the following characterization due to Fink, M{\'e}sz{\'a}ros, and St.~Dizier.  Given a diagram $D$, the \textbf{schubitope} $S_D$ is the convex hull of the set $\{\wt(C)\colon C\leq D\}$.  In particular, $S_D$ is the Newton polytope of $\chi_D$.

\subsection{Snow Diagrams}
To any diagram $D$, one can associate a diagram $\text{dark}(D)$ called the \textbf{dark cloud diagram} via the following process.  Iterate through the rows of $D$ from bottom to top.  In each row $i$, select the rightmost square $(i,j)$ of $D$ containing no element of $\text{dark}(D)$, and add $(i,j)$ to $\text{dark}(D)$.  If no such element exists in a given row, $\text{dark}(D)$ is left unchanged.
\begin{defnx}
    \label{def:snow-diagram}
    For any diagram $D$, the \textbf{snow diagram} is defined to be 
    \[\text{snow}(D):=D\cup\{(i,j):(i',j)\in\text{dark}(D),i'\geq i\}.\]
\end{defnx}
\begin{figure}
    \centering
    \scalebox{.7}{\begin{tikzpicture}
    \filldraw[draw=darkgray, color=green, opacity=.5] (0,3) rectangle (3,3.5);
    \filldraw[draw=darkgray, color=green, opacity=.5] (0,2) rectangle (2,2.5);
    \filldraw[draw=darkgray, color=green, opacity=.5] (0,.5) rectangle (.5,1.5);
    \draw[step=.5cm,black,very thin] (0,0) grid (4,4);
    \begin{scope}[xshift=5cm]
        \filldraw[draw=darkgray, color=green, opacity=.5] (0,3) rectangle (3,3.5);
    \filldraw[draw=darkgray, color=green, opacity=.5] (0,2) rectangle (2,2.5);
    \filldraw[draw=darkgray, color=green, opacity=.5] (0,1) rectangle (.5,2);
    \filldraw[draw=darkgray, color=green, opacity=.5] (0,2.5) rectangle (.5,3);
    \filldraw[draw=darkgray, color=green, opacity=.5] (0,3.5) rectangle (.5,4);
    \filldraw[draw=darkgray, color=green, opacity=.5] (1.5,2.5) rectangle (2,3);
    \filldraw[draw=darkgray, color=green, opacity=.5] (1.5,3.5) rectangle (2,4);
    \filldraw[draw=darkgray, color=green, opacity=.5] (2.5,3.5) rectangle (3,4);
    \filldraw[draw=darkgray, color=gray] (0,.5) rectangle (.5,1);
    \filldraw[draw=darkgray, color=gray] (1.5,2) rectangle (2,2.5);
    \filldraw[draw=darkgray, color=gray] (2.5,3) rectangle (3,3.5);
    \draw[step=.5cm,black,very thin] (0,0) grid (4,4);
    \end{scope}
    \end{tikzpicture}}
    \caption{Left: Skyline diagram for $\alpha=(0,6,0,4,0,1,1,0)$. Right: Snow diagram $\text{snow}(D_{Sky}(\alpha))$ with dark cloud diagram shown in gray.}
    \label{fig:skyline-snow}
\end{figure}
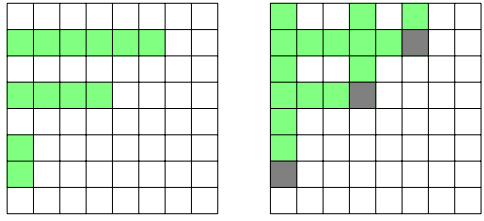

See Figure \ref{fig:skyline-snow} for an example of this construction.  A weak composition $\alpha$ is defined to be \textbf{snowy} if its positive entries are all distinct.  Pan and Yu \cite{pan2024top} show that snow diagrams connect to top-Lascoux polynomials in the following way.  
\begin{thm}[\cite{pan2024top}, Theorem 1.2, Proof of Lemma 4.5]
\label{thm:SnowDiagramLeading}
Let $\alpha$ be any weak composition.
    \begin{enumerate}
        \item The leading term of $\hat{\mathfrak{L}}_{\alpha}$ is a scalar multiple of $x^{\wt (\text{snow}(D_{Sky}(\alpha))}$.
        \item There exists a unique snowy weak composition $\gamma$ such that $\text{snow}(D_{Sky}(\alpha))=\text{snow}(D_{Sky}(\gamma))$.  For this $\gamma$, $\hat{\mathfrak{L}}_{\alpha}$ is a scalar multiple of $\hat{\mathfrak{L}}_{\gamma}$.
    \end{enumerate}
\end{thm}

Furthermore, Yu \cite{yu2023connection} shows the following.
\begin{thm}[\cite{yu2023connection}, Corollary 7.3]
\label{thm:SnowDiagramChar}
    For any weak composition $\alpha$, $\hat{\mathfrak{L}}_{\alpha}$ is a scalar multiple of $\chi_{\text{snow}(D_{Sky}(\alpha))}$.  
\end{thm}
Note that \cite{yu2023connection} states this result only for snowy weak compositions, but by Theorem \ref{thm:SnowDiagramLeading}, it also holds in this more general sense.

\subsection{Orthodontia}
In \cite{magyar1998schubert}, Magyar introduced the orthodontia algorithm on certain diagrams, yielding a formula for Schubert polynomials.  M{\'e}sz{\'a}ros, Setiabrata, and St.~Dizier \cite{meszaros2022orthodontia} extended Magyar's formula to give an orthodontia formula for Grothendieck polynomials.  We follow their exposition here.

Given a fixed $\%$-avoiding diagram $D \subset [n] \times [n]$, we will define $\mathbf{i}(D)=(i_1(D),\ldots,i_l(D))$, $\mathbf{k}(D)=(k_1(D),\ldots,k_n(D))$, and $\mathbf{m}(D)=(m_1(D),\ldots,m_l(D))$ via the following algorithm.  A \textbf{missing tooth} of a diagram $D$ is a square $(i,j)\notin D$ such that $(i+1,j)\in D$.  We say a column of $D$ is \textbf{packed} if it contains no missing teeth.  For every $a\in [n]$, let $k_a(D)$ equal the number of packed columns of $D$ which contain exactly $a$ squares (i.e. they contain precisely the top $a$ squares in the column).  Set $b=1$.  We then iterate the following steps.

(1) Replace all the packed columns with empty columns, and denote the resulting diagram as $D{\_}$.

(2) Let $(D{\_})_j$ be the leftmost non-empty column of $D{\_}$, and let $(i_b(D),j)$ be the topmost missing tooth in $(D{\_})_j$.  Let $^{i_b(D)}D$ be the diagram obtained by swapping rows $i_b(D)$ and $i_b(D)+1$ of $D{\_}$.  Set $m_b(D)$ to be the number of non-empty packed columns of $^{i_b(D)}D$.  Then increment $b$ by $1$, and repeat steps $(1)$ and $(2)$ using $^{i_b(D)}D$ in place of $D$.  Continue this process until no non-empty columns remain.

For a given $D$, the triple $(\mathbf{i}(D),\mathbf{k}(D),\mathbf{m}(D))$ is referred to as the \textbf{orthodontic sequence} of $D$.  Let $\omega_i$ denote $x_1\cdots x_i$.  Let $\overline{\pi}_i$ be the \textbf{Demazure-Lascoux operator} on polynomials of $x_1,\ldots,x_n$ defined by $\overline{\pi}_i(f)=\partial_i(x_i(1-x_{i+1})f)$.  Note that Demazure-Lascoux operators satisfy the following commutativity and braid relations:
\begin{enumerate}
    \item $\overline{\pi}_i\overline{\pi}_j=\overline{\pi}_j\overline{\pi}_i$ if $|i-j|>1$
    \item $\overline{\pi}_i\overline{\pi}_{i+1}\overline{\pi}_i=\overline{\pi}_{i+1}\overline{\pi}_i\overline{\pi}_{i+1}$
\end{enumerate}


\begin{defn}
    \label{def:OrthodontiaPolynomial}
    For any $\%$-avoiding $D$ with orthodontic sequence $\mathbf{i}(D)=(i_1(D),\ldots,i_l(D))$, $\mathbf{k}(D)=(k_1(D),\ldots,k_n(D))$, and $\mathbf{m}(D)=(m_1(D),\ldots,m_l(D))$, define
    \[\mathscr{G}_D := \omega_1^{k_1(D)}\cdots\omega_n^{k_n(D)}\overline{\pi}_{i_1(D)}(\omega_{i_1(D)}^{m_1(D)}\overline{\pi}_{i_2(D)}(\omega_{i_2(D)}^{m_2(D)}\cdots\overline{\pi}_{i_l(D)}(\omega_{i_l(D)}^{m_l(D)})\cdots)).
		\] \end{defn}

\begin{thm}[\cite{meszaros2022orthodontia},Theorem 1.1]
    \label{thm:GrothOrthodontia}
    For any $w\in S_n$, $\mathfrak{G}_w=\mathscr{G}_{D(w)}$.
\end{thm}

\begin{thm}[\cite{meszaros2022orthodontia},Section 7]
    \label{thm:LascOrthodontia}
    For any weak composition $\alpha$, $\mathfrak{L}_{\alpha}=\mathscr{G}_{D_{\text{Sky}}(\alpha)}$.  
\end{thm}


\begin{lem}[\cite{setiabrata2024double}, Lemma 4.5, Corollary 4.6]
    \label{lem:ortho_props}
    Suppose that $D$ is a column permutation of a skyline diagram, and let $\mathbf{i}(D)=(i_1(D),\ldots,i_l(D))$, $\mathbf{k}(D)=(k_1(D),\ldots,k_n(D))$, and $\mathbf{m}(D)=(m_1(D),\ldots,m_l(D))$ be the orthodontic sequence of $D$.  Then
    \begin{enumerate}
        \item Every diagram $^{i_b(D)}D$ in the orthodontic sequence of $D$ is a column permutation of a skyline diagram.
        \item If $k_b(D)\neq 0$, then $i_j(D) \neq b$ for all $j$.
        \item If $m_b(D) \neq 0$, then $i_j(D) \neq i_b(D)$ for all $j>b$.
        \item $\mathscr{G}_D := \overline{\pi}_{i_1(D)}(\overline{\pi}_{i_2(D)}(\cdots\overline{\pi}_{i_l(D)}(\omega_1^{k_1(D)}\cdots\omega_n^{k_n(D)}\omega_{i_1(D)}^{m_1(D)}\cdots\omega_{i_l(D)}^{m_l(D)})\cdots))$
    \end{enumerate}
\end{lem}

\begin{cor} \label{cor:ortho_permute}
    Suppose that $D$ is a column permutation of $D_{\text{Sky}}(\alpha)$ for some weak composition $\alpha$.  Then $\mathscr{G}_{D}=\mathscr{G}_{D_{\text{Sky}}(\alpha)}$.
\end{cor}
\begin{proof}
    Let $c$ be the leftmost column of $D$ such that column $c+1$ of $D$ contains more squares than column $c$.  Since $D$ is a column permutation of $D_{\text{Sky}}(\alpha)$, this implies that for all $r$ such that $(r,c)\in D$, $(r,c+1) \in D$.  Let $D'$ denote the diagram obtained by swapping columns $c$ and $c+1$ of $D$.  
    Observe that by definition, \[\omega_1^{k_1(D)}\cdots\omega_n^{k_n(D)}\omega_{i_1(D)}^{m_1(D)}\cdots\omega_{i_l(D)}^{m_l(D)}=\omega_1^{k_1(D')}\cdots\omega_n^{k_n(D')}\omega_{i_1(D')}^{m_1(D')}\cdots\omega_{i_{l'}(D')}^{m_{l'}(D')}.\]
    Using Lemma \ref{lem:ortho_props} and the commutation and braid relations satisfied by $\overline{\pi}_{i}$, it is then straightforward to show that $\mathscr{G}_{D}=\mathscr{G}_{D'}$.  By repeating this process, we see that $\mathscr{G}_{D}=\mathscr{G}_{D_{\text{Sky}}(\alpha)}$.
\end{proof}

\subsection{Bumpless Pipe Dreams}


A \textbf{bumpless pipe dream} (BPD) is a tiling of the $n \times n$ grid with the tiles  
\begin{center}
    \begin{tikzpicture}
\draw[gray, thin] (0,0) rectangle (.5,.5);
\draw[black, thick] (.25,0) -- (.25,.5);
\draw[black,thick] (0,.25) -- (.5,.25);
\draw[gray, thin] (1,0) rectangle (1.5,.5);
\draw[black, thick] (1.25,0) -- (1.25,.5);
\draw[gray, thin] (2,0) rectangle (2.5,.5);
\draw[black,thick] (2,.25) -- (2.5,.25);
\draw[gray, thin] (3,0) rectangle (3.5,.5);
\draw[black,thick] (3,.25) .. controls (3.25,.25) .. (3.25,.5);
\draw[gray, thin] (4,0) rectangle (4.5,.5);
\draw[black,thick] (4.25,0) .. controls (4.25,.25) .. (4.5,.25);
\draw[gray, thin] (5,0) rectangle (5.5,.5);
\end{tikzpicture}
\end{center} 
which forms a network of $n$ pipes, each running from the bottom to the right edge of the grid \cite{ls2021, weigandt2021}.  One can associate a permutation to each BPD by numbering the pipes $1$ to $n$ along the bottom edge of the grid and then reading off the labels on the right edge, ignoring any crossings after the first between each pair of pipes, i.e. replacing any redundant crossings with \scalebox{.6}{\begin{tikzpicture}
\draw[gray, thin] (0,0) rectangle (.5,.5);
\draw[black,thick] (0,.25) .. controls (.25,.25) .. (.25,.5);
\draw[black,thick] (.25,0) .. controls (.25,.25) .. (.5,.25);
\end{tikzpicture} } tiles.  A BPD is called \textbf{reduced} if every pair of pipes crosses at most once.  We denote the set of all BPDs associated to a permutation $w$ by $\text{BPD}(w)$ and the set of all reduced BPDs by $\text{RBPD}$.  

For a given $P \in \text{BPD}(w)$, we denote the set of blank tiles by $D(P)$ and the set of \scalebox{.6}{\begin{tikzpicture}
\draw[gray, thin] (0,0) rectangle (.5,.5);
\draw[black,thick] (0,.25) .. controls (.25,.25) .. (.25,.5);
\end{tikzpicture}} tiles by $U(P)$.  A \textbf{marked BPD} is a pair $(P,S)$ where $P$ is a BPD and $S\subset U(P)$.  We denote the set of marked BPDs associated to $w$ by $\text{MBPD}(w)$.  The \textbf{weight} of a marked BPD $(P,S)$ is the $n$-tuple $\text{wt}((P,S)):=(p_1,\ldots,p_n)$ where $p_i=|D(P) \cup S|$.  Let $\ell(w)$ denote the \textbf{length} of the permutation $w$, i.e. the total number of inversions in $w$.  
\begin{thm}[{\cite{weigandt2021}*{Corollary 1.5}}]
\label{thm:BPDGroth}
    For any $w\in S_n$,\[\mathfrak{G}_{w}=\sum_{(P,S) \in \text{MBPD}( w)}(-1)^{|D(P)| + |S| - \ell(w)}\mathbf{x}^{\text{wt}(P,S)}.\]
\end{thm}
There are certain moves which can be performed on while preserving the associated permutation.  Specifically, a \textbf{droop move} is a move of the form 
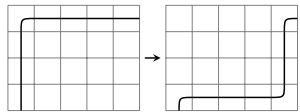
\begin{center}
\scalebox{.7}{\begin{tikzpicture}
\draw[step=.5cm,gray,very thin] (0,0) grid (2.5,2);
\draw[step=.5cm,gray,very thin] (2.999,0) grid (5.5,2);
\draw[-stealth, black, thick] (2.6,1) -- (2.9,1);
\draw[black, thick] (.25,0) -- (.25,1.5);
\draw[black, thick] (.25,1.5) .. controls (.25,1.75) .. (.5,1.75);
\draw[black,thick] (.5,1.75) -- (2.5,1.75);
\draw[black,thick] (3.25,0) .. controls (3.25,.25) .. (3.5,.25);
\draw[black,thick] (3.5,.25) -- (5,.25);
\draw[black,thick] (5, .25) .. controls (5.25,.25) .. (5.25,.5);
\draw[black,thick] (5.25,.5)--(5.25,1.5);
\draw[black,thick] (5.25,1.5).. controls (5.25,1.75) .. (5.5,1.75);
\end{tikzpicture}}
\end{center}
where the rectangle shown contains no other elbow tiles.  A \textbf{K-theoretic droop move} is a move of the form 
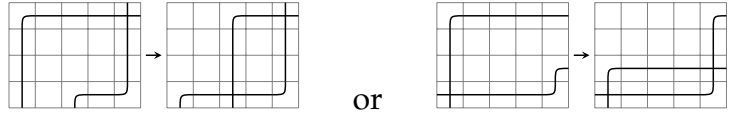
\begin{center}
    \scalebox{.7}{\begin{tikzpicture}
\draw[step=.5cm,gray,very thin] (0,0) grid (2.5,2);
\draw[step=.5cm,gray,very thin] (2.999,0) grid (5.5,2);
\draw[-stealth, black, thick] (2.6,1) -- (2.9,1);
\draw[black, thick] (.25,0) -- (.25,1.5);
\draw[black, thick] (.25,1.5) .. controls (.25,1.75) .. (.5,1.75);
\draw[black,thick] (.5,1.75) -- (2.5,1.75);
\draw[black,thick] (1.25,0) .. controls (1.25,.25) .. (1.5,.25) -- (2,.25) .. controls (2.25,.25) .. (2.25,.5) -- (2.25,2);

\draw[black,thick] (3.25,0) .. controls (3.25,.25) .. (3.5,.25);
\draw[black,thick] (3.5,.25) -- (5,.25);
\draw[black,thick] (5, .25) .. controls (5.25,.25) .. (5.25,.5);
\draw[black,thick] (5.25,.5)--(5.25,2);
\draw[black,thick] (4.25,0) -- (4.25,1.5) .. controls (4.25,1.75) .. (4.5,1.75) -- (5.5,1.75);
\end{tikzpicture}} \hspace{.5cm} or \hspace{.5cm}
    \scalebox{.7}{\begin{tikzpicture}
\draw[step=.5cm,gray,very thin] (0,0) grid (2.5,2);
\draw[step=.5cm,gray,very thin] (2.999,0) grid (5.5,2);
\draw[-stealth, black, thick] (2.6,1) -- (2.9,1);
\draw[black, thick] (.25,0) -- (.25,1.5);
\draw[black, thick] (.25,1.5) .. controls (.25,1.75) .. (.5,1.75);
\draw[black,thick] (.5,1.75) -- (2.5,1.75);
\draw[black,thick] (0,.25) -- (2,.25) .. controls (2.25,.25) .. (2.25,.5) .. controls (2.25,.75) .. (2.5,.75);

\draw[black,thick] (3,.25) -- (5,.25);
\draw[black,thick] (5, .25) .. controls (5.25,.25) .. (5.25,.5);
\draw[black,thick] (5.25,.5)--(5.25,1.5);
\draw[black,thick] (5.25,1.5).. controls (5.25,1.75) .. (5.5,1.75);
\draw[black,thick] (3.25,0) -- (3.25,.5) .. controls (3.25,.75) .. (3.5,.75) -- (5.5,.75);
\end{tikzpicture}}
\end{center}
where again, the rectangle shown contains no other elbow tiles.  For any $w$, the set $\text{RBPD}(w)$ is connected by droop moves \cite{ls2021}.  The set $\text{BPD}(w)$ is connected by droop moves and K-theoretic droop moves \cite{weigandt2021}.

\section{Bubbling Diagrams}
\label{sec:BubblingDiagrams}
In this section, we provide the basic machinery for working with bubbling diagrams, as introduced in \cite{Bubbling}.  Note that our notation and exposition differs somewhat from \cite{Bubbling} in order to fit the more generalized setting of the present paper.
\begin{defnx}
    A \textbf{dead square diagram} is an ordered triple $\mathcal{D} := (D,F,A)$ where $D$ is a diagram, $F \subset D$, and $A \subset D \setminus F$ such that for any square $(i,j) \in F$, there exists some $(i',j) \in A$ with $i' < i$ and $(i'',j) \notin D \setminus F$ for any $i' < i'' < i$. We refer to squares in $F$ as \textbf{dead squares}, squares in $A$ as \textbf{distinguished live squares}, and squares in $D \setminus F$ as \textbf{live squares}.
\end{defnx}

\begin{defnx}[Bubbling move]
\label{defn:bubbling-move}
Let $\mathcal D = (D, F, A)$ be a dead square diagram. Suppose that $(i,j)$ is a live square and that $(i-1,j)$ is an empty square.

Then, a \textbf{bubbling move} at $(i,j)$ produces the dead square diagram $\mathcal D' = (D', F', A')$ where:
\begin{align*}
D' &\colonequals D\setminus(i,j)\cup(i-1,j)\\
F' &\colonequals F\\
A' &\colonequals \begin{cases} A &\textup{ if } (i,j) \notin A \\ A \setminus (i,j)\cup(i-1,j)&\textup{ if } (i,j) \in A \end{cases}.
\end{align*}
In other words, we ``bubble up'' a live square from $(i,j)$ to $(i-1,j)$.
\end{defnx}

\begin{defnx}[K-Bubbling move]
\label{defn:k-bubbling-move}
Let $\mathcal D = (D, F, A)$ be a dead square diagram. Suppose that $(i,j)$ is a distinguished live square and that $(i-1,j)$ is an empty square.

Then, a \textbf{K-bubbling move} at $(i,j)$ produces the dead square diagram $\mathcal D' = (D', F', A')$ where:
\begin{align*}
D' &\colonequals D\cup(i-1,j)\\
F' &\colonequals F \cup (i,j)\\
A' &\colonequals A \setminus (i,j)\cup(i-1,j).
\end{align*}
In other words, we bubble up a distinguished live square from $(i,j)$ to $(i-1,j)$, leaving a dead copy behind at $(i,j)$.
\end{defnx}


\begin{figure}
    \centering
    \scalebox{.7}{\begin{tikzpicture}
     \filldraw[draw=darkgray, color=green, opacity=.5] (.5,1) rectangle (1,1.5);
     \filldraw[draw=darkgray, color=green, opacity=.5] (3.5,1) rectangle (4,1.5);
     \filldraw[draw=darkgray, color=green, opacity=.5] (6.5,1) rectangle (7,1.5);
     \filldraw[draw=darkgray, color=gold] (1,1) rectangle (1.5,1.5);
     \filldraw[draw=darkgray, color=gold] (4,1.5) rectangle (4.5,2);
     \filldraw[draw=darkgray, color=gold] (7,1.5) rectangle (7.5,2);
     \filldraw[draw=darkgray, color=lightgray] (7,1) rectangle (7.5,1.5);
        \draw[step=.5cm,black,very thin] (0,0) grid (2,2);
        \draw[step=.5cm,black,very thin] (2.999,0) grid (5,2);
        \draw[step=.5cm,black,very thin] (5.999,0) grid (8,2);
    \end{tikzpicture}}
    \caption{A dead square diagram (left) along with the result of performing a bubbling move (center) and a K-bubbling move (right).  Distinguished live squares are shown in gold and dead squares in gray.}
    \label{fig:bubblingmoves}
\end{figure}
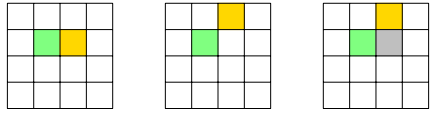

\begin{defnx}
    \label{def:SBD-general}
    Given a dead square diagram $\mathcal{D}=(D,F,A)$, a \textbf{streamlined bubbling diagram} with respect to $\mathcal{D}$ is any dead square diagram which can be obtained from $\mathcal{D}$ by a sequence of bubbling and K-bubbling moves.  We denote the set of streamlined bubbling diagrams by $\SBD(D,F,A)$.
\end{defnx}

The weight of a streamlined bubbling diagram $\mathcal{D}=(D,F,A)$ is $\wt(\mathcal{D}):=\wt(D)$.  The \textbf{excess} of $\mathcal D$ is the quantity $\text{ex}(\mathcal D):=|F|$. 

\begin{thm}[\cite{Bubbling}, Section 5]
    \label{thm:SBD_MConvex}
    Let $D$ be a diagram, and fix a set $A \subset D$ of distinguished live squares such that no column of $D$ contains more than one square in $A$.  Let $m=\max\{\text{ex}(\mathcal D):\mathcal{D} \in \SBD(D,\emptyset,A)\}$.  The set $\{z^{m-\text{ex}(\mathcal D)}\wt(\mathcal{D}):\mathcal{D}\in \SBD(D,\emptyset,A)\}$ is M-convex. 
\end{thm}


In \cite{Bubbling}, Hafner, Mészáros, Setiabrata, and St.~Dizier show that for any vexillary permutation $w \in S_n$, these diagrams can be used to compute the support of the Grothendieck polynomial $\mathfrak{G}_w$ in the following way.


\begin{defnx}
Let $w \in S_n$ be vexillary, and let $D(w)$ be the Rothe diagram of $w$.  We say that two squares $(i,j), (i',j')\in D(w)$ are \textbf{linked} if $i - i' = r_{D(w)}(i,j) - r_{D(w)}(i',j')$. A \textbf{linking class} is an equivalence class of linked squares.
\end{defnx}

\begin{defnx}
    Let $w\in S_n$ be a vexillary permutation.  The \textbf{left bubbling order} on the squares in $D(w)$ is the total ordering given by $(i,j) \prec_L (i',j')$ if:
\begin{enumerate}
\item $i > i'$, or

\item $i = i'$ and $(i,j)$ has fewer squares below it than $(i',j')$, or

\item $i = i'$, $(i,j)$ has the same number of squares below it as $(i',j')$ does, and $j < j'$.
\end{enumerate}
\end{defnx}

\begin{defnx}
    Let $w\in S_n$ be a vexillary permutation.  The \textbf{right bubbling order} on the squares in $D(w)$ is the total ordering given by $(i,j) \prec_R (i',j')$ if:
\begin{enumerate}
\item $i > i'$, or

\item $i = i'$ and $(i,j)$ has fewer squares below it than $(i',j')$, or

\item $i = i'$, $(i,j)$ has the same number of squares below it as $(i',j')$ does, and $j > j'$.
\end{enumerate}
\end{defnx}

Using either of these orderings, we may construct an ordered set of distinguished live squares in $D(w)$ according to the following procedure.  This set will be denoted $A_L(w)$ when constructed from the left bubbling order and $A_R(w)$ when constructed from the right bubbling order.  

\begin{enumerate}
\item Add the first square in the left bubbling order (resp. right bubbling order) to $A_L(w)$ (resp. $A_R(w)$).

\item Each subsequent square in the left bubbling order (resp. right bubbling order) will be appended to $A_L(w)$ (resp. $A_R(w)$) if and only if $A_L(w)$ (resp. $A_R(w)$) does not already contain a square in the same column and $A_L(w)$ (resp. $A_R(w)$) does not already contain a square in the same linking class.
\end{enumerate}


\begin{thm}[\cite{Bubbling}, Theorem 4.12]
\label{thm:supp-Gw-SBD}
If $w\in S_n$ is vexillary, then \[\supp(\mathfrak G_w) = \{\wt(\mathcal D)\colon \mathcal D \in \SBD(D(w), \emptyset, A_L(w))\}.\]
\end{thm}

\begin{remx}
    Note that in \cite{Bubbling}, $A_L(w)$ is simply denoted $A(w)$, and the set of streamlined bubbling diagrams $\SBD(D(w), \emptyset, A_L(w))$ is denoted $\SBD(w)$.  
\end{remx}

The same arguments as in \cite{Bubbling} can also be applied to show the following.

\begin{thm}
\label{thm:supp-Gw-SBD-right}
If $w\in S_n$ is vexillary, then \[\supp(\mathfrak G_w) = \{\wt(\mathcal D)\colon \mathcal D \in \SBD(D(w), \emptyset,A_R(w))\}.\]
\end{thm}

\begin{defnx}
\label{defn:Dtop}
    Given a diagram $D$, fix $A \subset D$ to be an arbitrary set of distinguished live squares such that no column of $D$ contains more than one square of $A$.  We define
    \[D_{\text{top}}(D,A):= D\cup\{(i,j):(i',j)\in A \text{ for some } i<i'\}.\]
For any vexillary $w\in S_n$, define $D_{\text{top}}(w):=D_{\text{top}}(D(w),A_L(w))$.
\end{defnx}

\begin{figure}
    \centering
    \scalebox{.7}{\begin{tikzpicture}
    \filldraw[draw=darkgray, color=green, opacity=.5] (.5,2.5) rectangle (2.5,3);
    \filldraw[draw=darkgray, color=green, opacity=.5] (1,1) rectangle (1.5,2);
    \filldraw[draw=darkgray, color=gold] (2,1) rectangle (2.5,1.5);
    \filldraw[draw=darkgray, color=gold] (1,1) rectangle (1.5,1.5);
    \filldraw[draw=darkgray, color=gold] (.5,2.5) rectangle (1,3);
        \draw[step=.5cm,black,very thin] (0,0) grid (3.5,3.5);
        \node at (1.25,1.25) {1};
        \node at (2.25,1.25) {2};
        \node at (1.25,1.75) {3};
        \node at (.75,2.75) {4};
        \node at (1.25,2.75) {7};
        \node at (1.75,2.75) {5};
        \node at (2.25,2.75) {6};
        \begin{scope}[xshift=4.5cm]
        \filldraw[draw=darkgray, color=green, opacity=.5] (.5,2.5) rectangle (2.5,3);
    \filldraw[draw=darkgray, color=green, opacity=.5] (1,1) rectangle (1.5,2);
    \filldraw[draw=darkgray, color=gold] (2,1) rectangle (2.5,1.5);
    \filldraw[draw=darkgray, color=gold] (1,1) rectangle (1.5,1.5);
    \filldraw[draw=darkgray, color=gold] (1.5,2.5) rectangle (2,3);
    \node at (1.25,1.25) {2};
        \node at (2.25,1.25) {1};
        \node at (1.25,1.75) {3};
        \node at (.75,2.75) {5};
        \node at (1.25,2.75) {7};
        \node at (1.75,2.75) {4};
        \node at (2.25,2.75) {6};
            \draw[step=.5cm,black,very thin] (0,0) grid (3.5,3.5);
        \end{scope}
        \begin{scope}[yshift=-4cm]
        \filldraw[draw=darkgray, color=green, opacity=.5] (.5,2.5) rectangle (2.5,3);
    \filldraw[draw=darkgray, color=green, opacity=.5] (1,1) rectangle (1.5,2);
    \filldraw[draw=darkgray, color=green, opacity=.5] (2,3) rectangle (2.5,3.5);
    \filldraw[draw=darkgray, color=gold] (2,2.5) rectangle (2.5,3);
    \filldraw[draw=darkgray, color=lightgray] (2,1) rectangle (2.5,2.5);
    \filldraw[draw=darkgray, color=green, opacity=.5] (1,3) rectangle (1.5,3.5);
    \filldraw[draw=darkgray, color=gold] (1,2) rectangle (1.5,2.5);
    \filldraw[draw=darkgray, color=lightgray] (1,1) rectangle (1.5,2);
    \filldraw[draw=darkgray, color=lightgray] (.5,2.5) rectangle (1,3);
    \filldraw[draw=darkgray, color=gold] (.5,3) rectangle (1,3.5);
            \draw[step=.5cm,black,very thin] (0,0) grid (3.5,3.5);
            \begin{scope}[xshift=4.5cm]
            \filldraw[draw=darkgray, color=green, opacity=.5] (.5,2.5) rectangle (2.5,3);
    \filldraw[draw=darkgray, color=green, opacity=.5] (1,1) rectangle (1.5,2);
    \filldraw[draw=darkgray, color=green, opacity=.5] (2,3) rectangle (2.5,3.5);
    \filldraw[draw=darkgray, color=gold] (2,2.5) rectangle (2.5,3);
    \filldraw[draw=darkgray, color=lightgray] (2,1) rectangle (2.5,2.5);
    \filldraw[draw=darkgray, color=green, opacity=.5] (1,3) rectangle (1.5,3.5);
    \filldraw[draw=darkgray, color=gold] (1,2) rectangle (1.5,2.5);
    \filldraw[draw=darkgray, color=lightgray] (1,1) rectangle (1.5,2);
    \filldraw[draw=darkgray, color=lightgray] (1.5,2.5) rectangle (2,3);
    \filldraw[draw=darkgray, color=gold] (1.5,3) rectangle (2,3.5);
            \draw[step=.5cm,black,very thin] (0,0) grid (3.5,3.5);
        \end{scope}
        \end{scope}
    \end{tikzpicture}}
    \caption{Top left: $D(w)$ for $w=1624735 $ showing the left bubbling order and $A_L(w)$. Top right: $D(w)$ showing the right bubbling order and $A_R(w)$.  Bottom left: $D_{\text{top}}(w)$.  Bottom right: $D_{\text{top}}(D(w),A_R(w))$.}
    \label{fig:bubbling-order-Dtop1}
\end{figure}
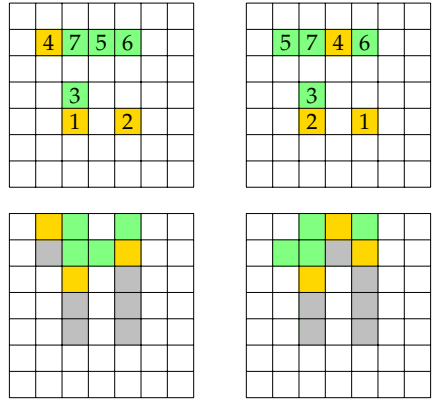

\begin{remx}
    \label{rem:LeftRightDTop}
    The diagrams $D_{\text{top}}(w)$ and $D_{\text{top}}(D(w),A_R(w))$ differ only by a permutation of their columns.  See Figure \ref{fig:bubbling-order-Dtop1} for an example.
\end{remx}

\begin{thm}[\cite{Bubbling}, Theorem 1.2]
    \label{thm:DTopGrothSupp}
    For vexillary $w \in S_n$, $\text{supp}(\hat{\mathfrak{G}}_w)=\text{supp}(\chi_{D_{\text{top}}(w)})$.  
\end{thm}

Furthermore, Hafner, Mészáros, Setiabrata, and St.~Dizier \cite{Bubbling} conjectured that for any vexillary permutation $w$, $\hat{\mathfrak{G}}_w$ equals an integer multiple of $\chi_{D_{\text{top}}(w)}$.  We prove this conjecture in Theorem \ref{thm:TopVexChar}.

Another family of permutations for which $\text{Supp}(\mathfrak{G}_w)$ can be computed by streamlined bubbling diagrams is fireworks permutations.  Given a diagram $D$, the \textbf{upwards closure} is the diagram $\overline{D}:=\{(i,j):(i',j)\in D \text{ for some }i'>i\}$.  Mészáros, Setiabrata, and St.~Dizier proved a characterization of $\text{Supp}(\hat{\mathfrak{G}}_w)$ for fireworks permutations.
\begin{thm}[\cite{mssd2022}, Theorem~3.14]
    \label{thm:firetop}
    For any fireworks $w\in S_n$, $\text{supp}(\hat{\mathfrak{G}}_w)=\{\wt(\overline{D(w)})\}.$
\end{thm}
In particular, this implies that for any firewroks $w$, $\Newton(\mathfrak{G}_w)$ is the schubitope $S_{\overline{D(w)}}$.  
Furthermore, Chou and Setiabrata proved the following.
\begin{thm}[\cite{chou2025newton}, Theorem~1.1, Corollary~1.2]
    \label{thm:fireMconv}
    For any fireworks $w\in S_n$, \[\supp(\mathfrak G_w) = \{\wt(\mathcal D)\colon \mathcal D \in \SBD(D(w), \emptyset,A\}\] where $A$ consists of the southmost square in each column of $D(w)$.  In particular, $\text{Supp}(\widetilde{{\mathfrak{G}}}_w)$ is M-convex.
\end{thm}

\section{New Bubbling Diagrams for Vexillary Grothendieck Polynomials}
\label{sec:NewvexBubb}
In this section, we introduce a new set of bubbling diagrams which also compute the supports of vexillary Grothendieck polynomials.
\begin{defnx}
    \label{defn:DistSquares-1PerRow}
    Let $w\in S_n$ be vexillary.  We define $A_{\text{new}}(w)\subset D(w)$ as follows:
    \begin{enumerate}
        \item Set $i=n$ and $A_{\text{new}}(w)=\emptyset$.
        \item Let $(i,j)$ be the first square of row $i$ (in the $\prec_R$ order) such that $(i',j)\notin A_{\text{new}}(w)$ for any $i'>i$, and add $(i,j)$ to $A_{\text{new}}(w)$.  If no such $(i,j)$ exists, leave $A_{\text{new}}(w)$ unchanged.
        \item If $i>1$, reduce $i$ by $1$ and repeat step (2).  Otherwise, terminate the procedure.
    \end{enumerate}
\end{defnx}
\begin{thm}
    \label{thm:SBD-1PerRow}
    If $w\in S_n$ is vexillary, then \[\supp(\mathfrak G_w) = \{\wt(\mathcal D)\colon \mathcal D \in \SBD(D(w), \emptyset,A_{\text{new}}(w))\}.\]
    Furthermore, $D_{\text{top}}(D(w),A_R(w))=D_{\text{top}}(D(w),A_{\text{new}}(w))$.
\end{thm}
See Figure \ref{fig:bubbling-order-Dtop} for an example.  In order to prove Theorem \ref{thm:SBD-1PerRow}, we will make use of the following lemmas.

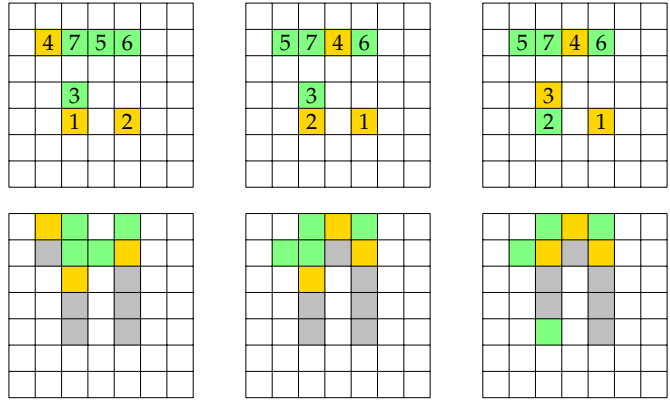
\begin{figure}
    \centering
    \scalebox{.7}{\begin{tikzpicture}
    \filldraw[draw=darkgray, color=green, opacity=.5] (.5,2.5) rectangle (2.5,3);
    \filldraw[draw=darkgray, color=green, opacity=.5] (1,1) rectangle (1.5,2);
    \filldraw[draw=darkgray, color=gold] (2,1) rectangle (2.5,1.5);
    \filldraw[draw=darkgray, color=gold] (1,1) rectangle (1.5,1.5);
    \filldraw[draw=darkgray, color=gold] (.5,2.5) rectangle (1,3);
        \draw[step=.5cm,black,very thin] (0,0) grid (3.5,3.5);
        \node at (1.25,1.25) {1};
        \node at (2.25,1.25) {2};
        \node at (1.25,1.75) {3};
        \node at (.75,2.75) {4};
        \node at (1.25,2.75) {7};
        \node at (1.75,2.75) {5};
        \node at (2.25,2.75) {6};
        \begin{scope}[xshift=4.5cm]
        \filldraw[draw=darkgray, color=green, opacity=.5] (.5,2.5) rectangle (2.5,3);
    \filldraw[draw=darkgray, color=green, opacity=.5] (1,1) rectangle (1.5,2);
    \filldraw[draw=darkgray, color=gold] (2,1) rectangle (2.5,1.5);
    \filldraw[draw=darkgray, color=gold] (1,1) rectangle (1.5,1.5);
    \filldraw[draw=darkgray, color=gold] (1.5,2.5) rectangle (2,3);
    \node at (1.25,1.25) {2};
        \node at (2.25,1.25) {1};
        \node at (1.25,1.75) {3};
        \node at (.75,2.75) {5};
        \node at (1.25,2.75) {7};
        \node at (1.75,2.75) {4};
        \node at (2.25,2.75) {6};
            \draw[step=.5cm,black,very thin] (0,0) grid (3.5,3.5);
        \end{scope}
        \begin{scope}[xshift=9cm]
        \filldraw[draw=darkgray, color=green, opacity=.5] (.5,2.5) rectangle (2.5,3);
    \filldraw[draw=darkgray, color=green, opacity=.5] (1,1) rectangle (1.5,2);
    \filldraw[draw=darkgray, color=gold] (2,1) rectangle (2.5,1.5);
    \filldraw[draw=darkgray, color=gold] (1,1.5) rectangle (1.5,2);
    \filldraw[draw=darkgray, color=gold] (1.5,2.5) rectangle (2,3);
    \node at (1.25,1.25) {2};
        \node at (2.25,1.25) {1};
        \node at (1.25,1.75) {3};
        \node at (.75,2.75) {5};
        \node at (1.25,2.75) {7};
        \node at (1.75,2.75) {4};
        \node at (2.25,2.75) {6};
            \draw[step=.5cm,black,very thin] (0,0) grid (3.5,3.5);
        \end{scope}
        \begin{scope}[yshift=-4cm]
        \filldraw[draw=darkgray, color=green, opacity=.5] (.5,2.5) rectangle (2.5,3);
    \filldraw[draw=darkgray, color=green, opacity=.5] (1,1) rectangle (1.5,2);
    \filldraw[draw=darkgray, color=green, opacity=.5] (2,3) rectangle (2.5,3.5);
    \filldraw[draw=darkgray, color=gold] (2,2.5) rectangle (2.5,3);
    \filldraw[draw=darkgray, color=lightgray] (2,1) rectangle (2.5,2.5);
    \filldraw[draw=darkgray, color=green, opacity=.5] (1,3) rectangle (1.5,3.5);
    \filldraw[draw=darkgray, color=gold] (1,2) rectangle (1.5,2.5);
    \filldraw[draw=darkgray, color=lightgray] (1,1) rectangle (1.5,2);
    \filldraw[draw=darkgray, color=lightgray] (.5,2.5) rectangle (1,3);
    \filldraw[draw=darkgray, color=gold] (.5,3) rectangle (1,3.5);
            \draw[step=.5cm,black,very thin] (0,0) grid (3.5,3.5);
            \begin{scope}[xshift=4.5cm]
            \filldraw[draw=darkgray, color=green, opacity=.5] (.5,2.5) rectangle (2.5,3);
    \filldraw[draw=darkgray, color=green, opacity=.5] (1,1) rectangle (1.5,2);
    \filldraw[draw=darkgray, color=green, opacity=.5] (2,3) rectangle (2.5,3.5);
    \filldraw[draw=darkgray, color=gold] (2,2.5) rectangle (2.5,3);
    \filldraw[draw=darkgray, color=lightgray] (2,1) rectangle (2.5,2.5);
    \filldraw[draw=darkgray, color=green, opacity=.5] (1,3) rectangle (1.5,3.5);
    \filldraw[draw=darkgray, color=gold] (1,2) rectangle (1.5,2.5);
    \filldraw[draw=darkgray, color=lightgray] (1,1) rectangle (1.5,2);
    \filldraw[draw=darkgray, color=lightgray] (1.5,2.5) rectangle (2,3);
    \filldraw[draw=darkgray, color=gold] (1.5,3) rectangle (2,3.5);
            \draw[step=.5cm,black,very thin] (0,0) grid (3.5,3.5);
        \end{scope}
        \begin{scope}[xshift=9cm]
        \filldraw[draw=darkgray, color=green, opacity=.5] (.5,2.5) rectangle (2.5,3);
    \filldraw[draw=darkgray, color=green, opacity=.5] (2,3) rectangle (2.5,3.5);
    \filldraw[draw=darkgray, color=gold] (2,2.5) rectangle (2.5,3);
    \filldraw[draw=darkgray, color=lightgray] (2,1) rectangle (2.5,2.5);
    \filldraw[draw=darkgray, color=green, opacity=.5] (1,3) rectangle (1.5,3.5);
    \filldraw[draw=darkgray, color=gold] (1,2.5) rectangle (1.5,3);
    \filldraw[draw=darkgray, color=lightgray] (1,1.5) rectangle (1.5,2.5);
    \filldraw[draw=darkgray, color=green, opacity=.5] (1,1) rectangle (1.5,1.5);
    \filldraw[draw=darkgray, color=lightgray] (1.5,2.5) rectangle (2,3);
    \filldraw[draw=darkgray, color=gold] (1.5,3) rectangle (2,3.5);
            \draw[step=.5cm,black,very thin] (0,0) grid (3.5,3.5);
        \end{scope}
        \end{scope}
    \end{tikzpicture}}
    \caption{Top left: $D(w)$ for $w=1624735 $ showing the left bubbling order and $A_L(w)$. Top center: $D(w)$ showing the right bubbling order and $A_R(w)$.  Top right: $D(w)$ showing the right bubbling order and $A_{\text{new}}(w)$.  Bottom left: $D_{\text{top}}(w)$.  Bottom center: $D_{\text{top}}(D(w),A_R(w))$.  Bottom right: $D_{\text{top}}(D(w),A_{\text{new}}(w))$.}
    \label{fig:bubbling-order-Dtop}
\end{figure}

\begin{lem}
\label{lem:RDSquares1}
    Fix a vexillary $w\in S_n$.  Suppose $(i,j_1),(i,j_2) \in D(w)$ are not linked and $(i,j_1)\prec_R (i,j_2)$.  Then $j_1>j_2$.
\end{lem}
\begin{proof}
Suppose (in search of a contradiction) that $j_1<j_2$.  By the definition of $\prec_R$, this would imply that there exists some $i'>i$ such that $(i',j_2)\in D(w)$ but $(i',j_1)\notin D(w)$.  This requires that $w^{-1}(j_1)<i'$.  
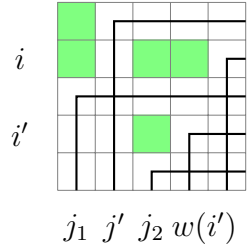
\begin{center}
\begin{tikzpicture}
\filldraw[green, opacity=0.5] (0,1.5) rectangle (0.5,2.5);
\filldraw[green, opacity=0.5] (1,1.5) rectangle (2,2);
\filldraw[green, opacity=0.5] (1,.5) rectangle (1.5,1);
    \draw[step=.5cm,gray,very thin] (0,0) grid (2.5,2.5);
    \node at (.25,-.5) {$j_1$};
    \node at (.75,-.5) {$j'$};
    \node at (1.25,-.5) {$j_2$};
    \node at (1.95,-.5) {$w(i')$};
    \node at (-.5,1.75) {$i$};
    \node at (-.5,.75) {$i'$};
    \draw[black, thick] (.25,0)--(.25,1.25)--(2.5,1.25);
    \draw[black, thick] (.75,0)--(.75,2.25)--(2.5,2.25);
    \draw[black, thick] (1.25,0)--(1.25,.25)--(2.5,.25);
    \draw[black, thick] (1.75,0)--(1.75,.75)--(2.5,.75);
    \draw[black, thick] (2.25,0)--(2.25,1.75)--(2.5,1.75);
\end{tikzpicture}
\end{center}
Furthermore, since $(i,j_1)$ and $(i,j_2)$ are not linked, there must exist some $j_1<j'<j_2$ such that $w^{-1}(j')<i$.  The sequence $j', j_1, w(i), j_2$ thus form a $2143$ pattern in the one-line notation of $w$, contradicting the fact that $w$ is vexillary.
\end{proof}
\begin{lem}
\label{lem:RDSquares2}
    Fix a vexillary $w\in S_n$.  Suppose $(i,j_1),(i,j_2) \in D(w)$ are not linked and $j_1<j_2$.  Let $m=r_{D(w)}(i,j_2)-r_{D(w)}(i,j_1)$.  Then for all $0\leq k\leq m$, $(i-k,j_1)\in D(w)$.  Furthermore, $(i-m,j_1)$ and $(i,j_2)$ are linked.
\end{lem}
\begin{proof}
    By definition, there must be $j_1<j'_1<\ldots<j'_m<j_2$ such that $w^{-1}(j'_k)<i$ for all $1\leq k\leq m$.  Furthermore, the properties of vexillary Rothe diagrams imply that for any $j'<j_1$, either $w^{-1}(j')<\text{min}\{w^{-1}(j'_k):1\leq k\leq m\}$ or $i<w^{-1}(j')$.  The result then follows from the definitions of Rothe diagrams and linking classes.
\end{proof}
\begin{lem}
    \label{lem:RDSquares3}
    Fix a vexillary $w\in S_n$.  Suppose $(i_1,j_1),(i_2,j_2)\in D(w)$ are linked and $i_1<i_2$.  Then $(i,j_1)\in D(w)$ for all $i_1\leq i \leq i_2$.
\end{lem}
\begin{proof}
    Observe that by the definitions of vexillary Rothe diagrams and linking classes, it must be the case that $j_1<j_2$ and for all $i_1\leq i<i_2$, $j_1<w(i)<j_2$.  The result follows immediately.  
\end{proof}
\begin{defnx}
    \label{def:NewDistSquares}
    Let $w\in S_n$ be vexillary.  Beginning with $A_R(w)$, we will construct a new set of distinguished live squares $A_R'(w)$ via the following procedure:
    \begin{enumerate}
        \item Set $i=n$ and $A'^{(n)}_R(w)=A_R(w)$.
        \item If $A'^{(i)}_R(w)$ has at most one square in row $i$, set $A'^{(i-1)}_R(w)=A'^{(i)}_R(w)$.

        \noindent Otherwise, let $\{(i,j_1),(i,j_2),\ldots,(i,j_m)\}$ be the set of squares in row $i$ of $A'^{(i)}_R(w)$, and suppose that $(i,j_1)\prec_R(i,j_2)\prec_R\cdots\prec_R(i,j_m)$.  Set $A'^{(i-1)}_R(w)=(A'^{(i)}_R(w)\setminus\{(i,j_2),\ldots,(i,j_m)\})\cup\{(i-1,j_2),\ldots,(i-1,j_m)\}$.
        \item If $i \geq 2$, decrease $i$ by $1$ and repeat step $2$.
        \item Define $A_R'(w):= A'^{(1)}_R(w)$.
    \end{enumerate}
\end{defnx}

\begin{lem}
    \label{lem:NewDiagramWellDef}
    For every vexillary $w \in S_n$, $A_R'(w)\subset D(w)$, and each row of $A_R'(w)$ contains at most one square.
\end{lem}
\begin{proof}
    By definition, $A'^{(n)}_R(w)\subset D(w)$.  Suppose that $1\leq i <n$ and $A'^{(i+1)}_R(w)\subset D(w)$.  If $A'^{(i+1)}_R(w)$ has at most one square in row $i+1$, then $A'^{(i)}_R(w)=A'^{(i+1)}_R(w) \subset D(w)$.  Otherwise, let $\{(i,j_1),(i,j_2),\ldots,(i,j_m)\}$ be the set of squares in row $i+1$ of $A'^{(i+1)}_R(w)$.  In this case, Lemmas \ref{lem:RDSquares1} and \ref{lem:RDSquares2} guarantee that $(i-1,j_2),\ldots,(i-1,j_m)\in D(w)$; therefore, $A'^{(i)}_R(w)\subset D(w)$.  We conclude by induction that $A'_R(w):= A'^{(1)}_R(w) \subset D(w)$.

    By definition, each of rows $2$ through $n$ must contain at most one square of $A_R'(w)$.  Furthermore, row $1$ of $A_R'(w)$ must contain at most one square; otherwise, Lemmas \ref{lem:RDSquares1} and \ref{lem:RDSquares2} would require that all but one square in row $1$ of $A_R'(w)$ have squares directly above them in $D(w)$, contradicting the fact that row $1$ is the top row.
\end{proof}

\begin{lem}
    \label{lem:NewDiagAlt}
Fix any vexillary $w \in S_n$ and $1\leq i_0 \leq n$.  Let $C \subset D(w)$ be any linking class which contains squares in row $i_0$.  If row $i_0$ of $C$ contains any square $(i_0,j)$ such that $(i',j)\notin A'^{(i_0)}_R$ for all $i<i'$, then the first such square $(i_0,j_0)$ in the $\prec_R$ ordering must satisfy $(i_0,j_0)\in A'^{(i_0)}_R$.
\end{lem}
\begin{proof}
    If every square $(i_0,j)$ in row $i_0$ of $C$ has a square of $A'^{(i_0)}_R$ below it in column $j$, then the result holds trivially.  Furthermore, if row $i_0$ of $C$ contains an element of $A_R(w)$, the result holds by construction.  We therefore suppose that neither of these conditions hold.  

    In this case, $C$ must contain a square $(i_0',j_0')$ from $A_R(w)$ with $i_0<i_0'$.  This square must be in a column strictly to the right of all squares in row $i_0$ of $C$.  Let $K$ denote the connected component of $D(w)$ containing the elements in row $i_0$ of $C$, and let $i_{\text{max}}$ be as large as possible such that rows $i_0$ and $i_{\text{max}}$ contain the same number of squares in $K$.  By Lemma \ref{lem:RDSquares3}, $i_0<i_0'\leq i_{\text{max}}$.  Let $i_1,\ldots, i_m$ denote the rows between $i_0$ and $i_{\text{max}}$ in which $K$ does not contain an element of $A_R(w)$.  We will suppose that $i_0<i_1<\ldots< i_m\leq i_{\text{max}}$.  Set $l=i_{\text{max}}-i_0$.

    Note that for each of $i_1,\ldots, i_m$, there must exist squares $(i_1',j_1'),\ldots, (i_m',j_m')\in A_R(w)\setminus K$ such that for each $1\leq a \leq m$ and every $(i_a,b)$ in $K$, $(i_a',j_a')$ is linked to $(i_a,b)$ and $(i_a',j_a')\prec_R(i_a,b)$.  By the definitions of linking classes and $\prec_R$, this implies that for all $1\leq a \leq m$, $i_a<i_a'$ and $(i-a',j_a')$ is in a column strictly to the right of all squares in row $i_a$ of $K$.  Observe that by Lemma \ref{lem:RDSquares3}, $i_1',\ldots,i_m' \leq i_{\text{max}}$.  Furthermore, observe that $(i_1',j_1'),\ldots, (i_m',j_m')$ are all distinct since they are in all in different linking classes.

    By Lemma \ref{lem:RDSquares1}, if $i_0<r\leq i_{\text{max}}$, $(r,j)\in K \cap A'^{(r)}_R$, and $(r,j'_a)\in A'^{(r)}_R$ for some $0\leq a \leq m$, then $(r,j'_a)\prec_R (r,j)$ and $(r,j) \notin A'^{(i_0)}_R$.  There must therefore be at least $m+1$ rows $i_0<r\leq i_{\text{max}}$ in which $K \cap A'^{(i_0)}_R$ is empty.  
    Note that $K\cap A'^{(i_0)}_R$ must contain at least $l-m=i_{\text{max}}-i-0-m$ squares between rows $i_0$ and $i_{\text{max}}$.  Specifically, at least the first $l-m=i_{\text{max}}-i-0-m$ squares (in the $\prec_R$ order) in row $i_0$ of $K$ which do not satisfy $(r,j)\in A'^{(i_0)}_R$ for any $i_{\text{max}}<r$ must satisfy $(r,j)\in A'^{(i_0)}_R$ for some $i_0\leq r \leq i_{\text{max}}$.  Since rows $i_0+1$ through $i_{\text{max}}$ can contain at most one square of $A'^{(i_0)}_R$ (by construction), at least one of these squares must be in row $i_0$ of $K \cap A'^{(i_0)}_R$.  By definition of $K$, this square is in the linking class $C$.
    \end{proof}

\begin{lem}
    \label{lem:pushing-up}
    Fix any vexillary $w \in S_n$, and fix a set of distinguished live squares $A \subset D(w)$.  Suppose that $(i,j) \in A$ and $(i-1,j)\in D(w)$, and let $A'=A-\{(i,j)\}\cup\{(i-1,j)\}$.  Then the set 
    $\{D \colon (D,F_1,A_1) \in \SBD(D(w),\emptyset,A) \text{ for some } F_1,A_1\}$ is identical to $\{D \colon (D,F_2,A_2) \in \SBD(D(w),\emptyset,A') \text{ for some $F_2$,$A_2$}\}$.
    Furthermore, $D_{\text{top}}(D,A)=D_{\text{top}}(D,A')$.
\end{lem}
\begin{proof}
Fix any $(D,F_1,A_1) \in \SBD(D(w),\emptyset,A)$.  If column $j$ of $F_1$ is empty, then $(D,F_1,A_1) \in \SBD(D(w),\emptyset,A')$.  Otherwise, let $(a,j)$ be the unique square in column $j$ of $A_1$, and let $(b,j)$ be the lowest square in column $j$ of $F_1$.  Let $(c,j)$ be the lowest square in column $j$ of $D$ such that $c<a$. Then $(D,F_2,A_2) \in \SBD(D(w),\emptyset,A')$ where $F_2=F_1\cup\{(a,j)\}\setminus\{(b,j)\}$ and $A_2=A_1\cup \{(c,j)\}\setminus\{(a,j)\}$. 

Conversely, fix any $(D,F_2,A_2) \in \SBD(D(w),\emptyset,A')$.  If column $j$ of $F_2$ is empty, then $(D,F_2,A_2) \in \SBD(D(w),\emptyset,A')$.  Otherwise, let $(a,j)$ be the highest square in column $j$ of $F_2$, and let $(b,j)$ be the highest square in column $j$ of $D\setminus F_2$ such that $a<b$.  Let $(c,j)$ be the unique square in column $j$ of $A_2$.  Then $(D,F_1,A_1) \in \SBD(D(w),\emptyset,A)$ where $F_1=F_2\cup\{(b,j)\}\setminus\{(a,j)\}$ and $A_1=A_2\cup \{(a,j)\}\setminus\{(c,j)\}$. 

The fact that $D_{\text{top}}(D,A)=D_{\text{top}}(D,A')$ follows directly from Definition \ref{defn:Dtop}.
\end{proof}

We are now ready to prove Theorem \ref{thm:SBD-1PerRow}.

\noindent\textit{Proof of Theorem \ref{thm:SBD-1PerRow}.} By Lemmas \ref{lem:NewDiagramWellDef} and \ref{lem:NewDiagAlt}, we can see that $A_{\text{new}}(w))=A'_{R}(w)$.  The result then follows immediately by Lemma \ref{lem:pushing-up} and the construction of $A'_{R}(w)$.
\qed

As a consequence of Theorem \ref{thm:SBD-1PerRow}, we now prove Theorem \ref{thm:TopVexChar}.  

\begin{lem}[\cite{pechenik2022k}]
    \label{lem:Vex-to-Las}
    For every vexillary $w\in S_n$, ${\mathfrak{G}}_w={\mathfrak{L}}_{\alpha}$ where $\alpha_i:=|\{j:i<j,w(i)>w(j)\}|$ for all $1\leq i \leq n$.
\end{lem}
\newtheorem*{thm:TopVexChar}{Theorem~\ref{thm:TopVexChar}}
\begin{thm:TopVexChar} 
    For any vexillary permutation $w$, $\hat{\mathfrak{G}}_w$ equals an integer multiple of $\chi_{D_{\text{top}}(w)}$.
\end{thm:TopVexChar}
\begin{proof}
    Fix some vexillary $w \in S_n$, and let $\alpha=(\alpha_1,\ldots,\alpha_n)$ where $\alpha_i:=|\{j:i<j,w(i)>w(j)\}|$ for all $1\leq i \leq n$.  Observe that by the properties of $D(w)$ for vexillary permutations, there exists a permutation of the columns of $D(w)$ which sorts the columns by inclusion (i.e. the columns are ordered such that for every $1<j\leq n$, every row with a square in column $j$ also has a square in column $j-1$).  Note that this permutation of the columns maps $D(w)$ to $D_{\text{Sky}}(\alpha)$.  By Definitions \ref{def:snow-diagram}, \ref{defn:Dtop}, and \ref{defn:DistSquares-1PerRow}, this same permutation of columns takes $D_{\text{top}}(D(w),A_{\text{new}}(w))$ to $\text{snow}(D_{\text{Sky}}(\alpha))$.

    By Theorem \ref{thm:SBD-1PerRow}, Lemma \ref{lem:pushing-up} and the construction of $A'_R$, we see that $D_{\text{top}}(D(w),A_{R}(w))$ differs from $\text{snow}(D_{\text{Sky}}(\alpha))$ by a permutation of columns.  Remark \ref{rem:LeftRightDTop} further implies that $D_{\text{top}}(w)$ differs from $\text{snow}(D_{\text{Sky}}(\alpha))$ by a permutation of columns.  By definition, $D_{\text{top}}(w)$ and $\text{snow}(D_{\text{Sky}}(\alpha))$ must therefore have the same dual character.  We conclude by Theorem \ref{thm:SnowDiagramChar} and Lemma \ref{lem:Vex-to-Las} that $\hat{\mathfrak{G}}_w$ is an integer multiple of $\chi_{D_{\text{top}}(w)}$.
\end{proof}

\section{Bubbling Diagrams for Other Permutations}
\label{sec:extendedBubb}
It is well known that vexillary permutations are precisely those whose Rothe diagrams are column-permutations of skyline diagrams \cite{manivel2001symmetric}.  Inspired by this, we define a generalization of vexillary permutations.  For a fixed $w\in S_n$, let $D_{\_}(w)$ denote the diagram obtained by replacing all packed columns of $D(w)$ with empty columns.  We say $w$ is \textbf{almost vexillary} if $D_{\_}(w)$ is a column-permutation of a skyline diagram for some weak composition.  Equivalently, almost vexillary permutations can be described by the following pattern avoidance condition.
\begin{thm}
    \label{thm:almostvex-pattern}
    A permutation is almost vexillary if and only if it avoids the patterns $13254$, $315264$, and $316254$.
\end{thm}
\begin{proof}
    If $w \in S_n$ contains a $13254$, $315264$, or $316254$ pattern, it is straightforward to show that $D(w)$ must contain some pair of rows $i_1,i_2$ and non-packed columns $j_1,j_2$ such that $(i_1,j_1),(i_2,j_2) \in D(w)$ but $(i_1,j_2),(i_2,j_1) \notin D(w)$.  This would imply that $w$ is not almost vexillary.

    Conversely, suppose that $w$ is not almost vexillary.  By definition, $D(w)$ must have rows $i_1,i_2$ and non-packed columns $j_1,j_2$ such that $(i_1,j_1),(i_2,j_2) \in D(w)$ but $(i_1,j_2),(i_2,j_1) \notin D(w)$.  Suppose without loss of generality that $j_1<j_2$, and note that by the properties of Rothe diagrams, we must also have $i_1<i_2$.  By the definition of $D(w)$, we know that $j_1<w(i_1)<j_2<w(i_2)$ and $i_1<w^{-1}(j_1)<i_2<w^{-1}(j_2)$.

    Since column $j_1$ is not packed, there must be some $i'$ such that $(i',j_1) \notin D(w)$ but $(i'+1,j_1) \in D(w)$.  Note that this implies $w(i')<j_1$ and $i'+1<w^{-1}(j_1)$.  If $i'<i_1$, then $w(i'),w(i_1),j_1,w(i_2),j_2$ is a $13254$ pattern in the one-line notation of $w$.  Otherwise, if $i_1<i'$, then because $(i'+1,j_1)\in D(w)$, we must have $j_1<w(i'+1)$.  In the latter case, we have $i_1<i'<i'+1<w^{-1}(j_1)<i_2<w^{-1}(j_2)$ and either $w(i')<j_1<w(i'+1)<w(i_1)<j_2<w(i_2)$, $w(i')<j_1<w(i_1)<w(i'+1)<j_2<w(i_2)$, $w(i')<j_1<w(i_1)<j_2<w(i'+1)<w(i_2)$, or $w(i')<j_1<w(i_1)<j_2<w(i_2)<w(i'+1)$.  In other words, $w$ contains at least one of the patterns $413265$, $314265$, $315264$, or $316254$.  Since, $413265$ and $314265$ both contain $13254$ patterns, we conclude that $w$ contains either a $13254$, $315264$, or $316254$ pattern.
\end{proof}
We show that certain results about vexillary permutations can be generalized to this broader family. 
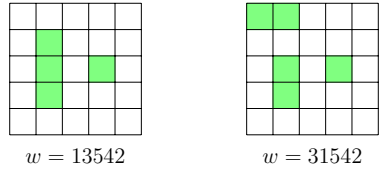
\begin{figure}
    \centering
    \scalebox{.7}{\begin{tikzpicture}
    \filldraw[draw=darkgray, color=green, opacity=.5] (.5,.5) rectangle (1,2);
    \filldraw[draw=darkgray, color=green, opacity=.5] (1.5,1) rectangle (2,1.5);
        \draw[step=.5cm,black,very thin] (0,0) grid (2.5,2.5);
        \node at (1.25,-.4) {$w=13542$};
        \begin{scope}[xshift=4.5cm]
        \filldraw[draw=darkgray, color=green, opacity=.5] (0,2) rectangle (1,2.5);
        \filldraw[draw=darkgray, color=green, opacity=.5] (.5,.5) rectangle (1,1.5);
        \filldraw[draw=darkgray, color=green, opacity=.5] (1.5,1) rectangle (2,1.5);
            \draw[step=.5cm,black,very thin] (0,0) grid (2.5,2.5);
            \node at (1.25,-.4) {$w=31542$};
        \end{scope}
    \end{tikzpicture}}
    \caption{Examples of vexillary (left) and almost vexillary (right) permutations with their Rothe diagrams.}
    \label{fig:permutations}
\end{figure}
\newtheorem*{thm:topdeg-almostvex}{Theorem~\ref{thm:topdeg-almostvex}}
\begin{thm:topdeg-almostvex}
    Let $w$ be any almost vexillary permutation.  Then $\hat{\mathfrak{G}}_w$ is a scalar multiple of $\chi_{{D}}$ for some diagram ${D}$.  In particular, $\text{Supp}(\hat{\mathfrak{G}}_w)$ is M-convex, and $\text{Newton}(\hat{\mathfrak{G}}_w)$ is a schubitope.  
\end{thm:topdeg-almostvex}
\begin{proof}
    Corollary \ref{cor:ortho_permute} states since $D_{\_}(w)$ is a column permutation of some skyline diagram $D_{\text{Sky}}(\alpha)$, then $\mathscr{G}_{D_{\_}(w)}=\mathscr{G}_{D_{\text{Sky}}(\alpha)}$ for some weak composition $\alpha$.  By Theorem \ref{thm:LascOrthodontia}, this implies that $\mathscr{G}_{D_{\_}(w)}=\mathfrak{L}_{\alpha}$.  Theorem \ref{thm:SnowDiagramChar} therefore implies that the maximal degree component of $\mathscr{G}_{D_{\_}(w)}$ is a scalar multiple of $\chi_{\text{snow}(D_{Sky}(\alpha))}$.

    For each $i\in [n]$, let $k_i$ be the number of packed columns of $D(w)$ with $i$ squares.  Let ${D}$ be the diagram obtained from $\text{snow}(D_{Sky}(\alpha))$ by appending $k_i$ packed columns with $i$ squares for every $i\in [n]$.  Observe that by definition, $\mathscr{G}_{D(w)} = \omega_1^{k_1}\cdots\omega_n^{k_n}\mathscr{G}_{D_{\_}(w)}$ where $w_i=x_1\cdots x_i$.  Furthermore, $\chi_{{D}}=w_1^{k_1}\cdots w_n^{k_n}\chi_{\text{snow}(D_{Sky}(\alpha))}$.  By Theorem \ref{thm:GrothOrthodontia}, we conclude that $\hat{\mathfrak{G}}_w$ is a scalar multiple of $\chi_{{D}}$.  
\end{proof}

\begin{conj}
    \label{conj:bubbling-almostvex}
    For every almost vexillary permutation $w$, there exists a subset $A\subset D(w)$ such that $\supp(\mathfrak G_w) = \{\wt(\mathcal D)\colon \mathcal D \in \SBD(D(w), \emptyset,A)\}.$
\end{conj}

Conjecture \ref{conj:bubbling-almostvex} would imply that for almost vexillary $w$, $\widetilde{\mathfrak G}_w$ has M-convex support.  
We next prove a set of rules which allow us to combine permutations whose Castelnuovo-Mumford polynomials have schubitope support into more permutations with the same property.

\begin{lem}
    \label{lem:symmetric-1k}
    Suppose $w\in S_n$ and $m\in[n-1]$ such that $w(k)=k$ for all $i\in[m]$.  Then $\mathfrak{G}_w$ is symmetric in the variables $x_1,\ldots,x_{m+1}$. 
\end{lem}
\begin{proof}
    Fix any $k\in[m]$.  We will define a bijection $f_k:\text{MBPD}(w)\rightarrow\text{MBPD}(w)$ as follows.  Consider a marked BPD $(P,S)$.  For each $i\in[n]$, let $P_i$ denote the set of tiles containing portions of pipe $i$.  Note that for each $i\in[k]$, pipe $i$ of $P$ must pass through both rows $k$ and $k+1$.  Let $A^k_i(P)$ denote the set of tiles $(k,j)$ such that $(k+1,j)\in P_i$.  Similarly, let $A^{k+1}_i(P)$ denote the set of tiles $(k+1,j)$ such that $(k,j)\in P_i$.  Recall that $D(P)$ denotes the set of blank tiles in $P$.  Denote by $S_{\text{row k}}$, the set of squares in row $k$ of $S$. Let
    \[a^k_i(P)=\begin{cases}
        |D(P)\cap A^k_i(P)| & \text{if } A^{k+1}_i(P)\cap P_j\neq\emptyset \text{ for some } j<k+1,i\neq j\\
        |D(P)\cap A^k_i(P)|+|P_i\cap S_{\text{row k}}| & \text{if } A^{k+1}_i(P)\cap P_j=\emptyset \text{ for all } j<k+1,i\neq j
    \end{cases}\]
    \[a^{k+1}_i(P)=\begin{cases}
        |D(P)\cap A^{k+1}_i(P)| & \text{if } A^{k}_i(P)\cap P_j\neq\emptyset \text{ for some } j<k+1,i\neq j\\
        |(D(P)\cup S)\cap A^{k+1}_i(P)| & \text{if } A^{k}_i(P)\cap P_j=\emptyset \text{ for all } j<k+1,i\neq j
    \end{cases}.\]
    Denote by $\text{Pairs}_k(P)$ the set of pairs $((k,j_1),(k+1,j_2))$ of \scalebox{.6}{\begin{tikzpicture}
\draw[gray, thin] (0,0) rectangle (.5,.5);
\draw[black,thick] (0,.25) .. controls (.25,.25) .. (.25,.5);
\end{tikzpicture}} tiles in $P$ such that for some $j$, $(k+1,j_2)\in P_j$ and $(k,j_1) \in A^k_j$..  The following example shows a section of rows $k$ and $k+1$ in a BPD $P$.  Pipe $i$ is shown in bold, and the sets $A^{k}_i(P)$ and $A^{k+1}_i(P)$ are shown in green and blue respectively.  The pair of tiles $((k,5),(k+1,6))$ is in $\text{Pairs}_k(P)$.

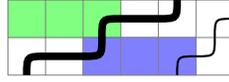
\begin{center}
    \begin{tikzpicture}
    \filldraw[draw=gray, color=green, opacity=.5] (0,.5) rectangle (1.5,1);
    \filldraw[draw=gray, color=blue, opacity=.5] (1,0) rectangle (2.5,.5);
        \draw[step=.5cm,gray,very thin] (0,0) grid (3,1);
        \draw[black, line width=.1cm] (0.25,0)..controls (0.25,0.25).. (0.5,0.25)--(1,0.25)..controls(1.25,0.25)..(1.25,0.5)..controls (1.25,.75) ..(1.5,.75)--(2,.75)..controls(2.25,.75)..(2.25,1);
        \draw[black, thick] (2.25,0)..controls(2.25,.25)..(2.5,.25)..controls(2.75,.25)..(2.75,.5)..controls(2.75,.75)..(3,.75);
    \end{tikzpicture}
\end{center}


    For each $(P,S)\in \text{MBPD}(w)$, we now construct $f_k((P,S))$ in the following way.  First, consider the set of all $((k,j_1),(k+1,j_2))\in \text{Pairs}_k(P)$ such that exactly one tile of the pair is in $S$.  Let $S'$ be the set obtained from $S$ by swapping each such pair so that the tile in each pair which was in $S$ is not in $S'$ and the tile in each pair which was not in $S$ is in $S'$.  Then, for every $i\in[k]$, compare $a^k_i(P)$ and $a^{k+1}_i(P)$.  If $a^{k}_i(P)=a^{k+1}_i(P)$, do not change the position of pipe $i$.  If $a^{k}_i(P)>a^{k+1}_i(P)$, reverse droop pipe $i$ into the $(a^{k}_i(P)-a^{k+1}_i(P))^{\text{th}}$ blank square of $A^k_i(P)$ from the right.  Similarly, if $a^{k}_i(P)<a^{k+1}_i(P)$, droop pipe $i$ into the $(a^{k+1}_i(P)-a^{k}_i(P))^{\text{th}}$ blank square of $A^{k+1}_i(P)$ from the left.  Note that in all of these cases, if pipe $i$ contained a tile from $S'$ in row $k$ of $k+1$, then pipe $i$ will still contain a \scalebox{.6}{\begin{tikzpicture}
\draw[gray, thin] (0,0) rectangle (.5,.5);
\draw[black,thick] (0,.25) .. controls (.25,.25) .. (.25,.5);
\end{tikzpicture}} tile in the same row after this move.  We add this new \scalebox{.6}{\begin{tikzpicture}
\draw[gray, thin] (0,0) rectangle (.5,.5);
\draw[black,thick] (0,.25) .. controls (.25,.25) .. (.25,.5);
\end{tikzpicture}} tile to $S'$.  By the properties of BPDs and the fact that $w(i)=i$ for every $i\in[k]$, the result of this procedure is a valid marked BPD $f_k((P,S))$.  Figure \ref{fig:symmertic-ex} shows an example of $f_k$ acting on a section of a marked BPD $(P,S)$.  By construction, $f_k((P,S))$ has the property that 
    \[\wt (f_k((P,S)))_i=\begin{cases}
        \wt((P,S))_{i+1} & \text{if } i=k \\
        \wt((P,S))_{i-1} & \text{if } i=k+1 \\
        \wt((P,S))_{i} & \text{otherwise} 
    \end{cases}.\]
    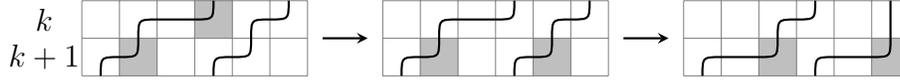
\begin{figure}
        \centering
        \begin{tikzpicture}
        \node at (-.5,.25) {$k+1$};
        \node at (-.5,.75) {$k$};
        \filldraw[lightgray] (.5,0) rectangle (1,.5);
        \filldraw[lightgray] (1.5,.5) rectangle (2,1);
            \draw[step=.5cm,gray,very thin] (0,0) grid (3,1);
            \draw[black,thick] (.25,0).. controls (.25,.25)..(.5,.25)..controls(.75,.25)..(.75,.5)..controls(.75,.75)..(1,.75)--(1.5,.75)..controls(1.75,.75)..(1.75,1);
            \draw[black,thick] (1.75,0)..controls (1.75,.25)..(2,.25)..controls(2.25,.25)..(2.25,.5)..controls(2.25,.75)..(2.5,.75)..controls(2.75,.75)..(2.75,1);
            \draw[-stealth, black, thick] (3.2,.5) -- (3.8,.5);
            \begin{scope}[xshift=4cm]
            \filldraw[lightgray] (.5,0) rectangle (1,.5);
        \filldraw[lightgray] (2,0) rectangle (2.5,.5);
                \draw[step=.5cm,gray,very thin] (0,0) grid (3,1);
                \draw[black,thick] (.25,0).. controls (.25,.25)..(.5,.25)..controls(.75,.25)..(.75,.5)..controls(.75,.75)..(1,.75)--(1.5,.75)..controls(1.75,.75)..(1.75,1);
            \draw[black,thick] (1.75,0)..controls (1.75,.25)..(2,.25)..controls(2.25,.25)..(2.25,.5)..controls(2.25,.75)..(2.5,.75)..controls(2.75,.75)..(2.75,1);
            \draw[-stealth, black, thick] (3.2,.5) -- (3.8,.5);
            \end{scope}
            \begin{scope}[xshift=8cm]
            \filldraw[lightgray] (1,0) rectangle (1.5,.5);
        \filldraw[lightgray] (2.5,0) rectangle (3,.5);
                \draw[step=.5cm,gray,very thin] (0,0) grid (3,1);
                \draw[black,thick] (.25,0).. controls (.25,.25)..(.5,.25)--(1,.25)
                ..controls(1.25,.25)..(1.25,.5)..controls(1.25,.75)..(1.5,.75)..controls(1.75,.75)..(1.75,1);
                \draw[black,thick] (1.75,0)..controls (1.75,.25)..(2,.25)--(2.5,.25)..controls(2.75,.25)..(2.75,.5)--(2.75,1);
            \end{scope}
        \end{tikzpicture}
        \caption{Construction of $f_k((P.S))$ in the proof of Lemma~\ref{lem:symmetric-1k} (gray shaded squares represent the set $S$).  We first swap the pair $((k,4),(k+1,5))$ and then droop both pipes.}
        \label{fig:symmertic-ex}
    \end{figure}

    Furthermore, $f_k$ is its own inverse, so in particular, $f_k$ is a bijection.  We can therefore conclude that $\mathfrak{G}_w$ is symmetric in the variables $x_k$ and $x_{k+1}$.  Since the above procedure holds for any $k\in[m]$, we see that $\mathfrak{G}_w$ is symmetric in the variables $x_1,\ldots,x_{k+1}$. 
\end{proof}
\begin{thm}
    \label{thm:layering-Groth}
    Let $w \in S_k$ and $w' \in S_n$ such that $w'(i)=i$ for all $i\in[k]$.  Let $w''$ be the permutation with one line notation $w(1)\ldots w(k) w'(k+1)\ldots w'(n)$.  Then $\mathfrak{G}_{w''}=\mathfrak{G}_w\mathfrak{G}_{w'}$.
\end{thm}
\begin{proof}

    Recall that by Theorem~\ref{thm:GrothOrthodontia}, $\mathfrak{G}_{w''}=\mathscr{G}_{D(w'')}$.  By defintion, there is some step in the orthodontia sequence for $D(w'')$ where the resulting diagram equals $D(w')$.  In particular, we can fix $j\leq l$ such that $i_1(D(w'')),\ldots,i_j(D(w''))<k$ and 
    \[\mathfrak{G}_{w''} = \omega_1^{k_1(D(w''))}\cdots\omega_n^{k_n(D(w''))}\overline{\pi}_{i_1(D(w''))}(\omega_{i_1(D(w''))}^{m_1(D(w''))}\cdots\overline{\pi}_{i_j(D(w''))}(\omega_{i_j(D(w''))}^{m_j(D(w''))}\mathfrak{G}_{w'})\cdots)).
		\]
    By Lemma~\ref{lem:symmetric-1k}, $\mathfrak{G}_{w'}$ commutes with $\overline{\pi}_{i_1(D(w''))},\ldots,\overline{\pi}_{i_j(D(w''))}$, so we have 
    \[\mathfrak{G}_{w''}  = \mathfrak{G}_{w'}\omega_1^{k_1(D(w''))}\cdots\omega_n^{k_n(D(w''))}\overline{\pi}_{i_1(D(w''))}(\omega_{i_1(D(w''))}^{m_1(D(w''))}\cdots\overline{\pi}_{i_j(D(w''))}(\omega_{i_j(D(w''))}^{m_j(D(w''))})\cdots)).
		\]
        Furthermore,
        \[\mathfrak{G}_{w}  = \omega_1^{k_1(D(w''))}\cdots\omega_n^{k_n(D(w''))}\overline{\pi}_{i_1(D(w''))}(\omega_{i_1(D(w''))}^{m_1(D(w''))}\cdots\overline{\pi}_{i_j(D(w''))}(\omega_{i_j(D(w''))}^{m_j(D(w''))})\cdots)),
		\]
        so we can conclude that $\mathfrak{G}_{w''}=\mathfrak{G}_w\mathfrak{G}_{w'}$.
\end{proof}
\begin{cor}
\label{cor:layering-SBD}
    Let $w \in S_k$ and $w' \in S_n$ such that $w'(i)=i$ for all $i\in[k]$.  Let $w''$ be the permutation with one line notation $w(1)\ldots w(k) w'(k+1)\ldots w'(n)$.  If there exist subsets $A\subset D(w)$ and $A'\subset D(w')$ such that $\supp(\mathfrak G_w) = \{\wt(\mathcal D)\colon \mathcal D \in \SBD(D(w), \emptyset,A)\}$ and $\supp(\mathfrak G_{w'}) = \{\wt(\mathcal D)\colon \mathcal D \in \SBD(D(w'), \emptyset,A')\}$, then $\supp(\mathfrak G_{w''}) = \{\wt(\mathcal D)\colon \mathcal D \in \SBD(D(w''), \emptyset,A\cup A')\}$. 
\end{cor}
\begin{proof}
    Follows directly from Theorem~\ref{thm:layering-Groth}.
\end{proof}
\begin{cor}
\label{cor:layering-Schub}
    Let $w \in S_k$ and $w' \in S_n$ such that $w'(i)=i$ for all $i\in[k]$.  Let $w''$ be the permutation with one line notation $w(1)\ldots w(k) w'(k+1)\ldots w'(n)$.  If both $\text{Newton}(\hat{\mathfrak{G}}_w)$ and $\text{Newton}(\hat{\mathfrak{G}}_{w'})$ are schubitopes, then $\text{Newton}(\hat{\mathfrak{G}}_{w''})$ is also a schubitope.
\end{cor}
\begin{proof}
    Follows directly from Theorem~\ref{thm:layering-Groth}.
\end{proof}

\begin{thm}
    \label{thm:layering-high-to-low}
    Let $w\in S_n$, and let $w'\in S_k$.  Let $w''$ denote the permutation $w'(1)+n,\ldots,w'(k)+n,w(1),\ldots,w(n)$.  Then $\mathfrak{G}_{w''}=\mathfrak{G}_w(x_{k+1},\ldots,x_{k+n})\mathfrak{G}_{w'}(x_1,\ldots,x_k)\prod_{i=1}^{k}x_i^n$.
\end{thm}
\begin{proof}
    Observe that every marked BPD for $w''$ has the following properties.  
    \begin{itemize}
        \item The northwest corner $\{(i,j:i\leq k, j\leq n)\}$ contains only blank tiles.
        \item The northeast corner $\{(i,j:i\leq k, j\geq n+1)\}$ is a marked BPD for $w'$.
        \item The southeast corner $\{(i,j:i\geq k+1, j\leq n)\}$ is a marked BPD for $w$.
        \item The southwest corner $\{(i,j:i\geq k+1, j\geq n+1)\}$ contains only cross tiles.
    \end{itemize}     
    Furthermore, any marked BPD with these properties is in $\MBPD(w'')$.  See Figure~\ref{fig:high-to-low} for an example.  The result then follows from Theorem~\ref{thm:BPDGroth}.  

    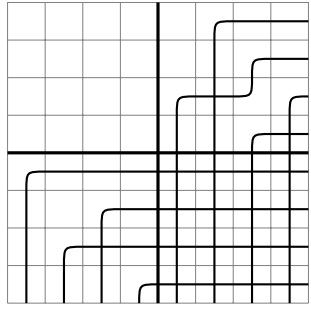
\begin{figure}
        \centering
        \begin{tikzpicture}
            \draw[step=.5cm,gray,very thin] (0,0) grid (4,4);
            \draw[black,very thick] (2,0)--(2,4);
            \draw[black,very thick] (0,2)--(4,2);
            \draw[black,thick] (3.75,0)--(3.75,2.5)..controls(3.75,2.75)..(4,2.75);
            \draw[black,thick] (3.25,0)--(3.25,2)..controls(3.25,2.25)..(3.5,2.25)--(4,2.25);
            \draw[black,thick] (2.75,0)--(2.75,3.5)..controls(2.75,3.75)..(3,3.75)--(4,3.75);
            \draw[black,thick] (2.25,0)--(2.25,2.5)..controls(2.25,2.75)..(2.5,2.75)--(3,2.75)..controls(3.25,2.75)..(3.25,3)..controls(3.25,3.25)..(3.5,3.25)--(4,3.25);
            \draw[black,thick] (1.75,0)--(1.75,0)..controls(1.75,.25)..(2,.25)--(4,.25);
            \draw[black,thick] (1.25,0)--(1.25,1)..controls(1.25,1.25)..(1.5,1.25)--(4,1.25);
            \draw[black,thick] (.75,0)--(.75,.5)..controls(.75,.75)..(1,.75)--(4,.75);
            \draw[black,thick] (.25,0)--(.25,1.5)..controls(.25,1.75)..(.5,1.75)--(4,1.75);
        \end{tikzpicture}
        \caption{A BPD for $w''$ as in the proof of Theorem~\ref{thm:layering-high-to-low}.  Here $w=1324$, $w'=2143$, and $w''=65871324$.}
        \label{fig:high-to-low}
    \end{figure}
\end{proof}
\begin{cor} \label{cor:high-to-low-SBD}
     Let $w\in S_n$, and let $w'\in S_k$.  Let $w''$ denote the permutation $w'(1)+n,\ldots,w'(k)+n,w(1),\ldots,w(n)$.  If $\text{Supp}({\mathfrak{G}}_w)=\{\text{wt}(\mathcal D): \mathcal D\in SBD(D(w),\emptyset,A\}$ and $\text{Supp}({\mathfrak{G}}_{w'})=\{\text{wt}(\mathcal D): \mathcal D\in SBD(D(w'),\emptyset,A'\}$ for some $A$ and $A'$, then there exists $A'' \subset D(w'')$ such that $\text{Supp}({\mathfrak{G}}_{w''})=\{\text{wt}(\mathcal D): \mathcal D\in SBD(D(w''),\emptyset,A''\}$.
\end{cor}
\begin{proof}
    Follows directly from Theorem~\ref{thm:layering-high-to-low}.
\end{proof}
\begin{cor} \label{cor:high-to-low-Schub}
     Let $w\in S_n$, and let $w'\in S_k$.  Let $w''$ denote the permutation $w'(1)+n,\ldots,w'(k)+n,w(1),\ldots,w(n)$.  If $\text{Newton}({\hat{\mathfrak{G}}}_w)$ and $\text{Newton}({\hat{\mathfrak{G}}}_{w'})$ are schubitopes, then so is $\text{Newton}({\hat{\mathfrak{G}}}_{w''})$.
\end{cor}
\begin{proof}
    Follows directly from Theorem~\ref{thm:layering-high-to-low}.
\end{proof}

As an application of Corollaries \ref{cor:layering-SBD} and \ref{cor:high-to-low-SBD}, we consider chains of fireworks and vexillary permutations.  For any $w \in S_n$, there exist $i_0=1<i_1<\cdots<i_{k+1}=n+1$ such that for every $j\in [k]$, $w(i_{j})\ldots w(i_{j+1}-1)$ is either a vexillary or fireworks permutation of some set of consecutive integers.  For every $j\in [k]$, let $\text{min}_j(w):=\text{min}\{w(i_{j}),\ldots,w(i_{j+1}-1)\}$.  We say $w\in S_n$ is a \textbf{dominant fireworks-vexillary chain} if there exists such a choice of $i_0=1<i_1<\cdots<i_{k+1}=n$ for which $\text{min}_1(w),\ldots,\text{min}_k(w)$ is $132$-avoiding.  
\begin{example}
    The permutation $w=769821534$ is a dominant fireworks-vexillary chain with $i_1=1$, $i_2=5$, and $i_3=7$.  Specifically, $7698$ is fireworks, $21$ is vexillary (and fireworks), and $534$ is also vexillary.  Observe that $\text{min}_1(w)=6$, $\text{min}_2(w)=1$, and $\text{min}_3(w)=3$; the sequence $613$ is $132$-avoiding.
\end{example}
\newtheorem*{thm:chains}{Theorem~\ref{thm:chains}}
\begin{thm:chains}
    If $w\in S_n$ is a dominant fireworks-vexillary chain, then there exists a subset $A\subset D(w)$ such that $\supp(\mathfrak G_w) = \{\wt(\mathcal D)\colon \mathcal D \in \SBD(D(w), \emptyset,A)\}$.  In particular, $\widetilde{\mathfrak{G}}_w$ has M-convex support, and $\text{Newton}(\hat{\mathfrak{G}}_w)$ is a schubitope.
\end{thm:chains}
\begin{proof}
    Observe that if $w'$ is a vexillary (resp. fireworks) permutation of $i_j,\ldots,i_{j+1}$, then $1,\ldots,i_j-1,w'(i_j),\ldots,w'(i_{j+1})$ is a vexillary (resp. fireworks) permutation of $1,\ldots, i_{j+1}$.  The result then follows from Theorems \ref{thm:supp-Gw-SBD} and \ref{thm:fireMconv} and Corollaries \ref{cor:layering-SBD} and \ref{cor:high-to-low-SBD}.
\end{proof}

\section*{Acknowledgements} 
The author would like to thank Sara Billey, Jack (Chen-An) Chou, Linus Setiabrata, and Tianyi Yu for helpful and inspiring conversations.

\bibliographystyle{amsplain}
\bibliography{refs}

\end{document}